\documentclass[12pt]{article}
\usepackage{latexsym,amsmath,amssymb}

\newif\ifdviwin

\usepackage[latin1]{inputenc}
\usepackage[english]{babel}
\usepackage{indentfirst}
\usepackage[mathscr]{eucal}
\usepackage{amssymb,amsmath,amsfonts}
\usepackage{fancybox,fancyhdr}
\usepackage{graphicx}
\usepackage{color}
\usepackage[export]{adjustbox}

\newif\ifdviwin

\dviwintrue

\def\ch{\mathfrak{h}}

\let\8=\infty \let\0=\emptyset

\let\alfa=\alpha

\let\parc=\partial
\let\ep=\varepsilon

\def\flecha{\rightarrow}
\def\esiz{\langle}
\def\esde{\rangle}

\def\cte.{\mathop{\rm cte.}\nolimits}

\def\div{\mathop{\rm div }\nolimits}

\def\Z{\mathbb{Z}}

\def\R{\mathbb{R}}

\def\cH{\mathcal{H}}

\def\cW{\mathcal{W}}

\def\S{\mathbb{S}}

 \newtheorem{defi}{Definition}
 \newtheorem{teo}[defi]{Theorem}
 \newtheorem{pro}[defi]{Proposition}
 \newtheorem{cor}[defi]{Corollary}
 \newtheorem{lem}[defi]{Lemma}

 \newenvironment{proof}{\rm \trivlist \item[\hskip \labelsep{\it
      Proof}:]}{\par\nopagebreak \hfill $\Box$ \endtrivlist}

\numberwithin{equation}{section}
\numberwithin{figure}{section}
\numberwithin{defi}{section}

\begin{document}

\mbox{}\vspace{0.4cm}
\begin{center}
\rule{14cm}{1.5pt}\vspace{0.5cm}

\renewcommand{\thefootnote}{\,}
{\Large \bf Rotational hypersurfaces of prescribed\\[0.2cm] mean curvature}\\ \vspace{0.5cm} {\large Antonio Bueno$^a$, José A.
Gálvez$^b$ and Pablo Mira$^c$}\\ \vspace{0.3cm} \rule{14cm}{1.5pt}
\end{center}
  \vspace{1cm}
$\mbox{}^a,^b$ Departamento de Geometría y Topología, Universidad de Granada,
E-18071 Granada, Spain. \\ e-mail: jabueno@ugr.es, jagalvez@ugr.es \vspace{0.2cm}

\noindent $\mbox{}^c$ Departamento de Matemática Aplicada y Estadística,
Universidad Politécnica de Cartagena, E-30203 Cartagena, Murcia, Spain. \\
e-mail: pablo.mira@upct.es  \let\thefootnote\relax\footnote{\hspace{-.75cm} Mathematics Subject
Classification: 53A10, 53C42, 34C05, 34C40}\vspace{0.3cm}

 \begin{abstract}
We use a phase space analysis to give some classification results for rotational hypersurfaces in $\R^{n+1}$ whose mean curvature is given as a prescribed function of its Gauss map. For the case where the prescribed function is an even function in $\S^n$, we show that a Delaunay-type classification holds for this class of hypersurfaces. We also exhibit examples showing that the behavior of rotational hypersurfaces of prescribed (non-constant) mean curvature is much richer than in the constant mean curvature case.
 \end{abstract}

\section{Introduction}

In this paper we study the existence and classification of rotational hypersurfaces $\Sigma$ of the Euclidean space $\R^{n+1}$ whose mean curvature $H_{\Sigma}$ is given as a prescribed function of its Gauss map $\eta:\Sigma\flecha \S^n\subset \R^{n+1}$. Specifically, given $\cH\in C^1(\S^n)$, we are interested in finding hypersurfaces $\Sigma$ that satisfy 

 \begin{equation}\label{presH}
H_{\Sigma}=\cH\circ \eta.
\end{equation}
Following \cite{BGM}, any such $\Sigma$ will be called a hypersurface of \emph{prescribed mean curvature} $\cH$, or an \emph{$\cH$-hypersurface}, for short. When $\Sigma$ is given as a graph $x_{n+1}=u(x_1,\dots, x_n)$ in canonical coordinates of $\R^{n+1}$, and we endow $\Sigma$ with its \emph{upwards orientation}, i.e., the one for which $\esiz \eta,e_{n+1}\esde >0$, where $e_{n+1}=(0,\dots,0,1)$, then \eqref{presH} is written as the elliptic, second order quasilinear PDE 
\begin{equation}\label{eqH}
{\rm div}\left(\frac{Du}{\sqrt{1+|Du|^2}}\right) = n \cH (Z_u), \hspace{1cm} Z_u:= \frac{(-Du,1)}{\sqrt{1+|Du|^2}},
\end{equation}
where $\div, D$ denote respectively the divergence and gradient operators on $\R^n$; note that $Z_u$ is the unit normal of the graph. A difficulty arising in the study of \eqref{eqH} is that, in contrast with the well-studied case in which the right-hand side in \eqref{eqH} depends only on $u$, this equation does not have in general a variational structure.

The geometric study of $\cH$-hypersurfaces is motivated by classical works of Alexandrov and Pogorelov in the 1950s (see \cite{Al, Po}) about existence and uniqueness of ovaloids in $\R^{n+1}$ defined by a prescribed curvature function in terms of its Gauss map. 
However, the global geometry of complete, non-compact $\cH$-hypersurfaces in $\R^{n+1}$ has not been studied in the literature for general choices of $\cH$ until recently. Indeed, the only well studied geometric theories in this regard are those of constant mean curvature hypersurfaces (corresponding to $\cH={\rm constant})$ and of self-translating solitons of the mean curvature flow, which correspond to the choice $\cH(x)=\esiz x,v\esde$ for some $v\in \S^n$.

In \cite{BGM} the authors started the development of the global theory of complete $\cH$-hypersurfaces, with a special emphasis on its relation with the case of constant mean curvature hypersurfaces. We showed there that, under mild symmetry and regularity assumptions on the function $\cH$, the theory of $\cH$-hypersurfaces sometimes admits a uniform treatment that resembles the constant mean curvature case.

Our goal in the present paper is to complement the theory developed in \cite{BGM} with a study of the rotational $\cH$-hypersurfaces in $\R^{n+1}$, in the case that the function $\cH \in C^1(\S^n)$ is rotationally symmetric, i.e. $\cH(x)=\ch(\esiz x, e_{n+1}\esde)$ for some $\ch\in C^1([-1,1])$. In such generality for $\cH$, it seems hopeless to find an explicit description of such rotational $\cH$-hypersurfaces, similar to the cases of CMC hypersurfaces, or of the self-translating solitons of the mean curvature flow. Thus, we will follow a different approach. We will treat the resulting ODE as a nonlinear autonomous system and we will carry out a qualitative study of its solutions through a phase space analysis. In this way, we will prove that, for some very general choices of $\cH$, the geometry of these rotational $\cH$-hypersurfaces resembles the classical constant mean curvature case. However, we will also show that in general there exists an immense variety of global geometric behaviors for general rotational $\cH$-hypersurfaces. As a matter of fact, some of the examples that we construct here show that the hypotheses in some theoretical results in \cite{BGM} are necessary.

We next explain the organization of the paper. In Section \ref{sec:rot1} we will develop the phase space analysis explained above, and use it to prove that any rotational $\cH$-hypersurface in $\R^{n+1}$ diffeomorphic to $\S^n$ is strictly convex (Theorem \ref{teoes}). Let us remark that by a \emph{strictly convex hypersurface} in $\R^{n+1}$ we mean an oriented hypersurface $\Sigma$ in $\R^{n+1}$ all of whose principal curvatures are different from zero and of the same sign. In particular, if $\Sigma$ is compact and strictly convex, then the Gauss map $\eta:\Sigma\flecha \S^n$ is a global diffeomorphism, and $\Sigma$ is the boundary of a strictly convex bounded set of $\R^{n+1}$; we call then $\Sigma$ a \emph{strictly convex sphere}.

When $\cH$ vanishes at some point, we will construct in Section \ref{sec:bowls}, for very general choices of $\cH$, a family of rotational $\ch$-bowls (which are entire strictly convex graphs) and of $\ch$-catenoids (which resemble the usual minimal catenoids in $\R^{n+1}$, and are diffeomorphic to $\S^{n-1}\times \R$).

In Section \ref{sec:rot2} we will prove a classification theorem for rotational $\cH$-hypersurfaces in $\R^{n+1}$, in the case that $\cH\in C^1(\S^n)$ is positive, rotationally symmetric and \emph{even} (i.e. $\cH(x)=\cH(-x)>0$ for every $x\in \S^n$). In these general conditions, we will show in Theorem \ref{dela} that the geometry of such rotational $\cH$-hypersurfaces follows the same pattern as the classical Delaunay classification of rotational hypersurfaces of non-zero constant mean curvature in $\R^{n+1}$. That is, all such examples are convex spheres, right circular cylinders, properly embedded hypersurfaces of \emph{unduloid type}, or proper, non-embedded hypersurfaces of \emph{nodoid type}.

In contrast with this classification theorem and the existence of bowls and catenoids, in Section \ref{sec:nocmc} we will show that there exist many rotational $\cH$-surfaces in $\R^3$ which do not behave at all like CMC surfaces in $\R^3$. For instance, we will consruct complete, convex $\cH$-graphs converging to a cylinder, or properly embedded disks asymptotically wiggling around a cylinder. We will also show examples with a \emph{wing-like} shape, or with two strictly convex ends pointing in opposite directions. Many of these examples can also be constructed so that they self-intersect. All this variety, just for the very particular class of rotational $\cH$-surfaces in $\R^3$, shows that the class of $\cH$-hypersurfaces in $\R^{n+1}$ is indeed very large, and rich in what refers to possible examples and geometric behaviors.

The phase space analysis developed in the present paper has motivated some studies in more general situations, see \cite{B1,B2,B3,GM3}. Let us also remark that the study of rotational graphs $x_{n+1}=u(x_1,\dots, x_n)$ for which the mean curvature is given as a prescribed function of $u$ (and not of its Gauss map) has received many contributions, in $\R^{n+1}$ and other ambient spaces; see e.g. \cite{BV,DRT,DG,Lo2,P} and references therein for just a few examples. See also \cite{K,KN} for the case where the mean curvature is prescribed as a function of the profile curve of the rotational hypersurface.

\vspace{0.3cm}

{\bf Acknowledgements:} This work is part of the PhD thesis of the first author.

\section{Phase space analysis of rotational $\cH$-hypersurfaces}\label{sec:rot1}

In this section we will let $\cH\in C^1(\S^n)$ be a rotationally symmetric function, i.e. 
$\cH(x)= \mathfrak{h}(\esiz x,v\esde)$
for some $v\in \S^n$ and some $C^1$ function $\mathfrak{h}$ on $[-1,1]$. Up to an Euclidean change of coordinates, we will assume that $v=e_{n+1}$, and so
 \begin{equation}\label{presim0}
\cH(x)= \mathfrak{h}(\esiz x,e_{n+1}\esde).
 \end{equation} 
Thus, equation \eqref{presH} for an immersed oriented hypersurface $\Sigma$ in $\R^{n+1}$ yields \begin{equation}\label{presim}H_{\Sigma}=\mathfrak{h}\circ \nu,\end{equation} where $\nu:=\langle \eta,e_{n+1}\rangle$ is the \emph{angle function} of $\Sigma$.

Let $\Sigma$ be an immersed, oriented, rotational hypersurface in $\R^{n+1}$, obtained as the orbit of a regular planar curve parametrized by arc-length $$\alfa(s)=(x(s),0,...,0,z(s)):I\subset \R\flecha \R^{n+1}, \hspace{1cm} x(s)>0,$$ under the action of all orientation preserving linear isometries of $\R^{n+1}$ that leave the $x_{n+1}$-axis pointwise fixed. A parametrization for $\Sigma$ is $$\psi(s,p)= (x(s) p, z(s)): I\times \S^{n-1} \flecha \R^{n+1}.$$ By changing the orientation of the profile curve if necessary, the angle function of $\Sigma$ is given by $\nu=x'(s)$. There are at most two different principal curvatures on $\Sigma$, given by
\begin{equation}\label{pricu}\kappa_1=\kappa_{\alfa}= x'(s) z''(s)-x''(s) z'(s), \hspace{1cm} \kappa_{2}=\cdots = \kappa_n= \frac{z'(s)}{x(s)},
\end{equation} where $\kappa_{\alfa}$ denotes the geodesic curvature of the profile curve $\alfa(s)$.

Let now $\cH,\mathfrak{h}$ be in the conditions stated at the beginning of this section, related by \eqref{presim0}, and let $\Sigma$ be a rotational $\cH$-hypersurface in $\R^{n+1}$. Thus, $\Sigma$ satisfies \eqref{presim}. So, from \eqref{pricu}, the profile curve $\alfa(s)$ of $\Sigma$ satisfies \begin{equation}\label{ode1}
n\mathfrak{h}(x')=x'z''-x''z'+(n-1)\frac{z'}{x}.
\end{equation}
Noting that $x'^2+z'^2=1$, we obtain from \eqref{ode1} that $x(s)$ is a solution to the autonomous second order ODE
\begin{equation}\label{ode3}
x''=(n-1)\frac{1-x'^2}{x}-n\varepsilon \,\mathfrak{h}(x')\sqrt{1-x'^2}, \hspace{1cm} \varepsilon ={\rm sign}(z'),
\end{equation}
on every subinterval $J\subset I$ where $z'(s)\neq 0$ for all $s\in J$.

Denoting $x'=y$, \eqref{ode3} transforms into the first order autonomous system
\begin{equation}\label{1ordersys}
\left(\begin{array}{c}
x\\
y
\end{array}\right)'=\left(\begin{array}{c}
y\\
 (n-1)\frac{\displaystyle{1-y^2}}{\displaystyle{x}}-n\varepsilon\, \mathfrak{h}(y)\sqrt{1-y^2}
\end{array}\right)=:F(x,y).
\end{equation}
The phase space of \eqref{1ordersys} is $\Theta_{\ep}:=(0,\8)\times (-1,1)$, with coordinates $(x,y)$ denoting, respectively, the distance to the rotation axis and the angle function of $\Sigma$. If $\ep \mathfrak{h}(0)>0$, there is a unique equilibrium of \eqref{1ordersys} in $\Theta_{\ep}$, namely 
 \begin{equation}\label{equil} 
 e_0:=\left(\frac{n-1}{n\ep \mathfrak{h}(0)},0\right).
 \end{equation}
 This equilibrium corresponds to the case where $\Sigma$ is a right circular cylinder $\S^{n-1}(r)\times \R$ in $\R^{n+1}$ of constant mean curvature $\mathfrak{h}(0)$ and vertical rulings. Otherwise, there are no equilibria in $\Theta_{\ep}$. The orbits $(x(s),y(s))$ provide then a foliation by regular proper $C^1$ curves of $\Theta_{\ep}$ (or of $\Theta_{\ep}-\{e_0\}$, in case $e_0$ exists). It is important to observe that, since $\ch$ is $C^1$, the uniqueness of the initial value problem for \eqref{1ordersys} implies that if an orbit $(x(s),y(s))$ converges to $e_0$, the value of the parameter $s$ goes to $\pm \8$.
 
The points in $\Theta_{\ep}$ where $y'(s)=0$ are those placed at the intersection of $\Theta_{\ep}$ with the (possibly disconnected) horizontal graph given for $y\in [-1,1]$ by 
\begin{equation}\label{graga}
x=\Gamma_{\ep}(y)=\frac{(n-1)\sqrt{1-y^2}}{n \varepsilon \, \mathfrak{h}(y)}.
\end{equation}
Note that $\Gamma_{\ep}(y)$ does not take a finite value at the zeros of $\mathfrak{h}(y)$, since $\mathfrak{h}$ is $C^1$. We will denote $\Gamma_{\ep}:=\Theta_{\ep} \cap \{x=\Gamma_{\ep}(y)\}$. It must be remarked that $\Gamma_{\ep}$ might be empty; for instance, in the case $\mathfrak{h}\leq 0$ and $\ep=1$. A computation shows that the values $s\in J$ where the profile curve $\alfa(s)=(x(s),z(s))$ of $\Sigma$ has zero geodesic curvature are those where $y'(s)=0$, i.e., those where $(x(s),y(s))\in \Gamma_{\ep}$.

\begin{figure}[h]
\begin{center}
\includegraphics[width=.65\textwidth]{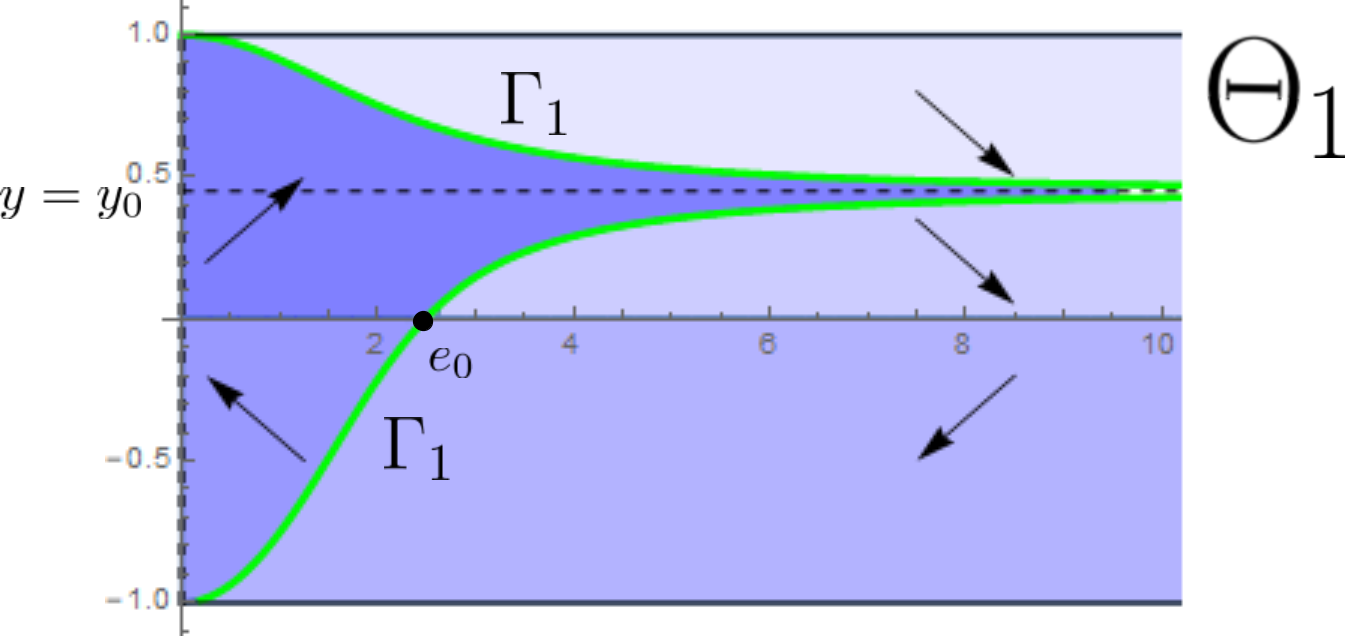}
\end{center}
\caption{An example of phase space $\Theta_1$ for some choice of $\mathfrak{h}\geq 0$ with $\mathfrak{h}(y_0)=0$. The poing $e_0$ is the equilibrium. The curve $\Gamma_1$ has two connected components, and there exist five monotonicity regions, separated by $\Gamma_1$ and the axis $y=0$. Each arrow indicates the monotonicity direction at each of those regions.}
\label{figesfa}
\end{figure}

The curve $\Gamma_{\ep}$ and the axis $y=0$ divide $\Theta_{\ep}$ into connected components where both $x(s)$ and $y(s)$ are monotonous. In particular, at each of these \emph{monotonicity regions}, the profile curve $\alfa(s)$ has geodesic curvature of constant sign. Specifically, by \eqref{pricu} we have at each point $\alfa(s)$, $s\in J$:
 \begin{equation}\label{signk}
 {\rm sign} (\kappa_1)= {\rm sign} (-\ep y'(s)), \hspace{0.5cm} {\rm sign}(\kappa_i)= \ep, \ i=2,\dots, n.
 \end{equation}
Also, by viewing the orbits of \eqref{1ordersys} as graphs $y=y(x)$ wherever possible (i.e. wherever $y\neq 0$), we have
\begin{equation}\label{yfuncx}
y\frac{dy}{dx}=(n-1)\frac{1-y^2}{x}-n\varepsilon\,\mathfrak{h}(y)\sqrt{1-y^2}.
\end{equation}
Thus, in each of these monotonicity regions the sign of the quantity $y y'(x)$ is constant. In particular, the behavior of the orbit of \eqref{1ordersys} passing through a given point $(x_0,y_0)\in \Theta_{\ep}$ is determined by the signs of $y_0$ and $x_0-\Gamma_{\ep}(y_0)$ (wherever $\Gamma_{\ep}(y_0)$ exists). We point out below some trivial particular consequences:

\begin{lem}\label{profas}
In the above conditions, for any $(x_0,y_0)\in \Theta_{\ep}$ such that $\Gamma_{\ep}(y_0)$ exists, the following properties hold:
\begin{enumerate}
\item If $x_0>\Gamma_{\ep}(y_0)$ (resp. $x_0<\Gamma_{\ep}(y_0)$) and $y_0>0$, then $y(x)$ is strictly decreasing (resp. increasing) at $x_0$. 
\item If $x_0>\Gamma_{\ep}(y_0)$ (resp. $x_0<\Gamma_{\ep}(y_0)$) and $y_0<0$, then $y(x)$ is strictly increasing (resp. decreasing) at $x_0$.
\item If $y_0=0$, then the orbit passing through $(x_0,0)$ is orthogonal to the $x$ axis.
\item If $x_0=\Gamma_{\ep}(y_0)$, then $y'(x_0)=0$ and $y(x)$ has a local extremum at $x_0$.

\end{enumerate}
\end{lem}

In order to describe further properties of orbits in the phase space $\Theta_{\ep}$, we turn our attention to the Dirichlet problem. Recall that if $\Sigma$ is an upwards-oriented graph $x_{n+1}=u(x_1,\dots, x_n)$ of prescribed mean curvature $\cH\in C^1(\S^n)$, then $u$ satisfies \eqref{eqH}. The next result follows from \cite[Corollary 1]{Mar}, and gives general conditions for the existence of $\cH$-graphs in $\R^{n+1}$. 

\begin{pro}\label{dah}
Let $\Omega\subset \R^n$ be a bounded $C^{2,\alfa}$ domain and $\varphi \in C^{2,\alfa}(\overline{\Omega})$ for some $\alfa\in (0,1)$. Let $\cH\in C^1(\S^n)$, and assume that:
 \begin{enumerate}
 \item
${\rm max}_{\S^n} |\cH|^n \, {\rm vol} (\Omega) <\omega_n,$ where $\omega_n$ stands for the volume of the $n$-dimensional unit ball.
 \item
$H_{\parc \Omega} (x)\geq \frac{n}{n-1} |\cH(\nu(x))|$ for all $x\in \parc\Omega$, where $\nu$ is the inner pointing unit conormal of $\parc \Omega$, viewed as a vector in $\S^n\subset\R^{n+1}$, and $H_{\parc \Omega}$ is the mean curvature of the submanifold $\parc\Omega$ with respect to $\nu$.
 \end{enumerate}
Then, the Dirichlet problem 
 \begin{equation}\label{dirh}
 \left\{ \begin{array}{lll} {\rm div}\left(\displaystyle \frac{Du}{\sqrt{1+|Du|^2}}\right) = n \cH (Z_u), \hspace{0.7cm} Z_u:= \displaystyle\frac{(-Du,1)}{\sqrt{1+|Du|^2}}, & \text{ in } & \Omega, \\ u=\varphi & \text{ on } & \parc \Omega, \end{array}\right.
 \end{equation}
has a unique solution $u\in C^{2,\alfa}(\overline{\Omega})$.
\end{pro}

A direct consequence from Proposition \ref{dah} is:

\begin{lem}\label{canoex}
Let $\cH\in C^1(\S^n)$ be in the conditions stated at the beginning of the section, and $\delta \in \{-1,1\}$. Then, there exists a unique (up to vertical translations) $\cH$-hypersurface $\Sigma$ in $\R^{n+1}$ that is rotational with respect to the the $x_{n+1}$-axis, and that meets this rotation axis orthogonally at some point $p\in \Sigma$, with unit normal at $p$ given by the vertical unit vector $\delta e_{n+1}\in \R^{n+1}$.
\end{lem}
\begin{proof}
By Proposition \ref{dah}, we can solve the Dirichlet problem \eqref{dirh} for the equation of upwards-oriented $\cH$-graphs in $\R^{n+1}$, on a sufficiently small ball $\Omega\subset \R^n$, with constant Dirichlet data $\varphi$ on the boundary. Since $\cH$ is rotationally invariant, the graph $\Sigma$ of the solution is a rotational $\cH$-hypersurface in $\R^{n+1}$ with unit normal $e_{n+1}$ at the origin. Since the equation in \eqref{dirh} is invariant by additive constants, the uniqueness of $\Sigma$ in these conditions is immediate from the maximum principle. The same argument can be done for \emph{downwards-oriented} $\cH$-graphs in $\R^{n+1}$, what completes the proof of Lemma \ref{canoex}.
\end{proof}
Lemma \ref{canoex} has the following consequence for our analysis of the phase space $\Theta_{\ep}$.

\begin{cor}\label{ejefase}
Assume that $\mathfrak{h}(\delta)\neq 0$ for some $\delta \in \{-1,1\}$, and let $\ep\in \{-1,1\}$ such that $\ep \mathfrak{h}(\delta)>0$. Then, there exists a unique orbit in $\Theta_{\ep}$ that has $(0,\delta)\in \overline{\Theta_{\ep}}$ as an endpoint. There is no such an orbit in $\Theta_{-\ep}$.
\end{cor}
\begin{proof}
Let $\Sigma$ be the rotational $\cH$-hypersurface given for $\delta$ by Lemma \ref{canoex} (the relation between $\cH$ and $\mathfrak{h}$ is given by \eqref{presim0}).  Let $\alfa(s)=(x(s),z(s))$ be the profile curve of $\Sigma$, defined for $s\in [0,s_0)$ or $s\in (-s_0,0]$ depending on the orientation chosen on $\alfa$, and assume that $x(0)=z'(0)=0$, i.e. $s=0$ corresponds to the point $p_0$ of orthogonal intersection of $\Sigma$ with its rotation axis. Since $p_0$ is an umbilical point of $\Sigma$, all principal curvatures of $\Sigma$ at $p_0$ are equal and of the same sign as $\mathfrak{h}(\delta)$.

By \eqref{pricu}, the geodesic curvature of $\alfa(s)$ at $s=0$ is non-zero, and thus the sign of $z'(s)$, which will as usual be denoted by $\ep$, is constant for $s$ small enough. It follows then again by \eqref{pricu} that $\ep \mathfrak{h}(\delta)>0$. Consequently, the profile curve $\alfa(s)$ generates an orbit in the phase space $\Theta_{\ep}$ with $(0,\delta)$ as an endpoint. Uniqueness of this orbit follows from the uniqueness in Lemma \ref{canoex}. It is also clear from the argument that such an orbit cannot exist in $\Theta_{-\ep}$, because of the condition $\ep \mathfrak{h}(\delta)>0$.
\end{proof}

The previous corollary can be used to describe geometrically the compact, rotational $\cH$-hypersurfaces immersed in $\R^{n+1}$ that are diffeomorphic to $\S^n$. 

\begin{teo}\label{teoes}
Let $\Sigma$ be an immersed, rotational $\cH$-hypersurface in $\R^{n+1}$ diffeomorphic to $\S^n$, for $\cH\in C^1(\S^n)$ rotationally symmetric. Then $\Sigma$ is a strictly convex sphere.
\end{teo}

\begin{proof}
By Proposition 2.6 in \cite{BGM}, the existence of a compact $\cH$-hypersurface $\Sigma$ for some $\cH\in C^1(\S^n)$ implies that $\cH$ never vanishes. Thus, up to a change of orientation, we can assume that $\cH>0$ in our situation.

Let $\alfa(s)=(x(s),z(s))$ be the profile curve of $\Sigma$, and let $\mathfrak{h}\in C^1([-1,1])$ be given by \eqref{presim0}. Since $\mathfrak{h}>0$ by the previous argument, the curve $\Gamma_{\ep}:=\Theta_{\ep}\cap \{x=\Gamma_{\ep} (y)\}$ where $\Gamma_{\ep}(y)$ is given by \eqref{graga} does not exist for $\ep=-1$, and for $\ep=1$, $\Gamma:=\overline{\Gamma_{\ep}}$ is a compact connected arc with endpoints $(0,1)$ and $(0,-1)$. Hence, we have four monotonicity regions $\Lambda_1,\dots,\Lambda_4$ inside $\Theta_1$, with monotonicities given by Lemma \ref{profas}, and an equilibrium $e_0$; see Figure \ref{figestre}. 

\begin{figure}[h]
\begin{center}
\includegraphics[width=.55\textwidth]{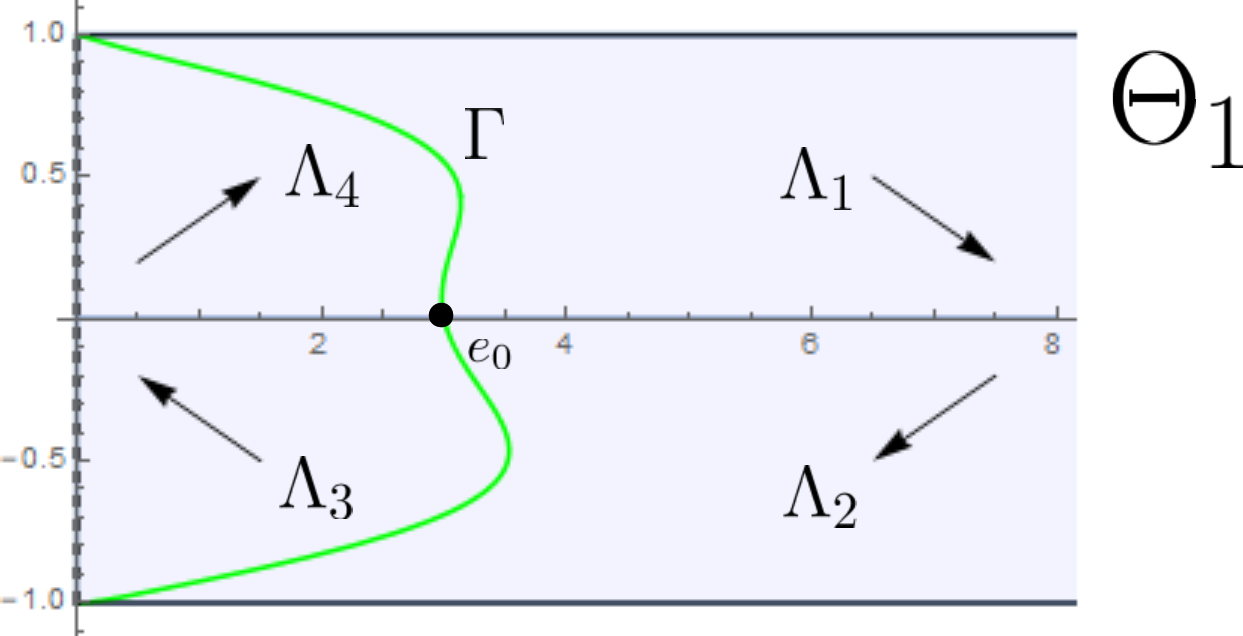}
\end{center}
\caption{The phase space $\Theta_1$, showing the monotonicity direction of each region $\Lambda_i$.}
\label{figestre}
\end{figure}

By Corollary \ref{ejefase}, there exists an orbit $\gamma$ in the phase space $\Theta_1$ that has $(0,1)$ as an endpoint; it corresponds to an open subset of $\Sigma$ that contains a point of $\Sigma$ with unit normal $e_{n+1}$, at which $\Sigma$ meets its rotation axis orthogonally. By the monotonicity properties of the regions $\Lambda_i$, the curve $\gamma$ lies in the region $\Lambda_1$ for all points near $(0,1)$.
By the regularity and compactness of $\Sigma$, the curve $\gamma$ cannot approach any point of the form $(0,y)$ with $y\in (-1,1)$. Indeed, if $\gamma$ had such an endpoint, $\Sigma$ would be asymptotic to its rotation axis (this corresponds to $y=0$), or it would present a non-removable isolated singularity around some point touching its rotation axis ($y\neq 0$), and these situations cannot happen. Also, it cannot happen that $\gamma$ converges asymptotically to the equilibrium $e_0$, since in that case, as explained after \eqref{equil}, the parameter $s$ would converge to $\pm \8$, and this 
contradicts the compactness of $\Sigma$. Since $\gamma$ cannot self-intersect, it becomes then clear from the monotonocity properties of the phase space that there are only two possible behaviors for $\gamma$:
 \begin{enumerate}
 \item[i)] If $\gamma$ intersects $\Gamma$ at some point in $\Theta_1$, then $\gamma$ enters at some moment in the region $\Lambda_3$ (and hence, afterwards it also enters in $\Lambda_4$), and so it has to converge asymptotically to the equilibrium $e_0$ of $\Theta_1$, which is not possible, as explained above.
  \item[ii)] Otherwise, $\gamma$ stays in $\Lambda_1 \cup \Lambda_2\cup \{y=0\}$, and it is a graph of the form $x=g(y)>0$, with $y\in (y_0,1)$ for some $y_0\in [-1,1)$. Again by compactness of $\Sigma$, we must have $y_0=-1$. Thus, $\gamma$ can be extended to be a compact graph $x=g(y)\geq 0$ for $y\in [-1,1]$, and it has a second endpoint at some $(x_1,-1)$ with $x_1\geq 0$.
  \end{enumerate}

Next, we can repeat all the argument above, and obtain a second orbit $\sigma$ in $\Lambda_1\cup \Lambda_2\cup \{y=0\} \subset \Theta_1$, that has $(0,-1)$ as an endpoint, and which corresponds to a piece of $\Sigma$ where it intersects its rotation axis with unit normal $-e_{n+1}$. We conclude that $\sigma$ can be extended to be a graph $x=t(y)\geq 0$ for $y\in [-1,1]$, with a second endpoint at some $(x_2,1)$ with $x_2\geq 0$. Since $\gamma$ and $\sigma$ cannot intersect on $\Theta_1$, the only possibility is that $x_1=0$ or $x_2=0$. Thus, by the uniqueness property of Corollary \ref{ejefase}, we have $\gamma=\sigma$, which is then an orbit in $\Theta_1$ joining $(0,1)$ with $(0,-1)$. 

By \eqref{1ordersys}, and since $\gamma$ stays in the region $\Lambda_1\cup \Lambda_2\cup \{y=0\}$, it follows that $y'(s)<0$ for all $s$. Consequently, by \eqref{signk} we see that all principal curvatures of $\Sigma$ are positive at every point, and so $\Sigma$ is a strictly convex sphere in $\R^{n+1}$.

\end{proof}

Theorem \ref{teoes} does not hold if the function $\cH\in C^1(\S^n)$ is merely assumed to be in $C^0(\S^n)$. Indeed, let $S$ be a smooth strictly convex radial graph in $\R^{n+1}$ over a closed ball $\overline{B}_{\rho}\subset \R^n$ of a certain radius $\rho>0$, with $\parc S =\parc \overline{B}_{\rho}\times \{0\}.$ Let $S_t$ be obtained from $S$ by a symmetry with respect to the $x_{n+1}=0$ hyperplane and a translation of vector $(0,\dots, 0,t)$, $t>0$, and let $C_t$ denote the compact (with boundary) vertical cylinder $\parc \overline{B}_{\rho}\times [0,t]$ in $\R^{n+1}$. Assume that $\Sigma_t:=S\cup C_t \cup S_t$ defines a smooth rotational hypersurface in $\R^{n+1}$. Then, it is easy to see that if two points $p_1,p_2\in \Sigma_t$ have the Gauss map image in $\S^n$, they also have the same mean curvature at those points. That is, $\Sigma_t$ satisfies equation \eqref{presH} for some rotationally symmetric $\cH\in C^0(\S^n)$. Clearly, $\Sigma_t$ is a rotational $\cH$-hypersurface that bounds a convex set of $\R^{n+1}$, but that is not strictly convex, see Figure \ref{esfnoconvex}.

\begin{figure}[h]
\begin{center}
\includegraphics[width=.25\textwidth]{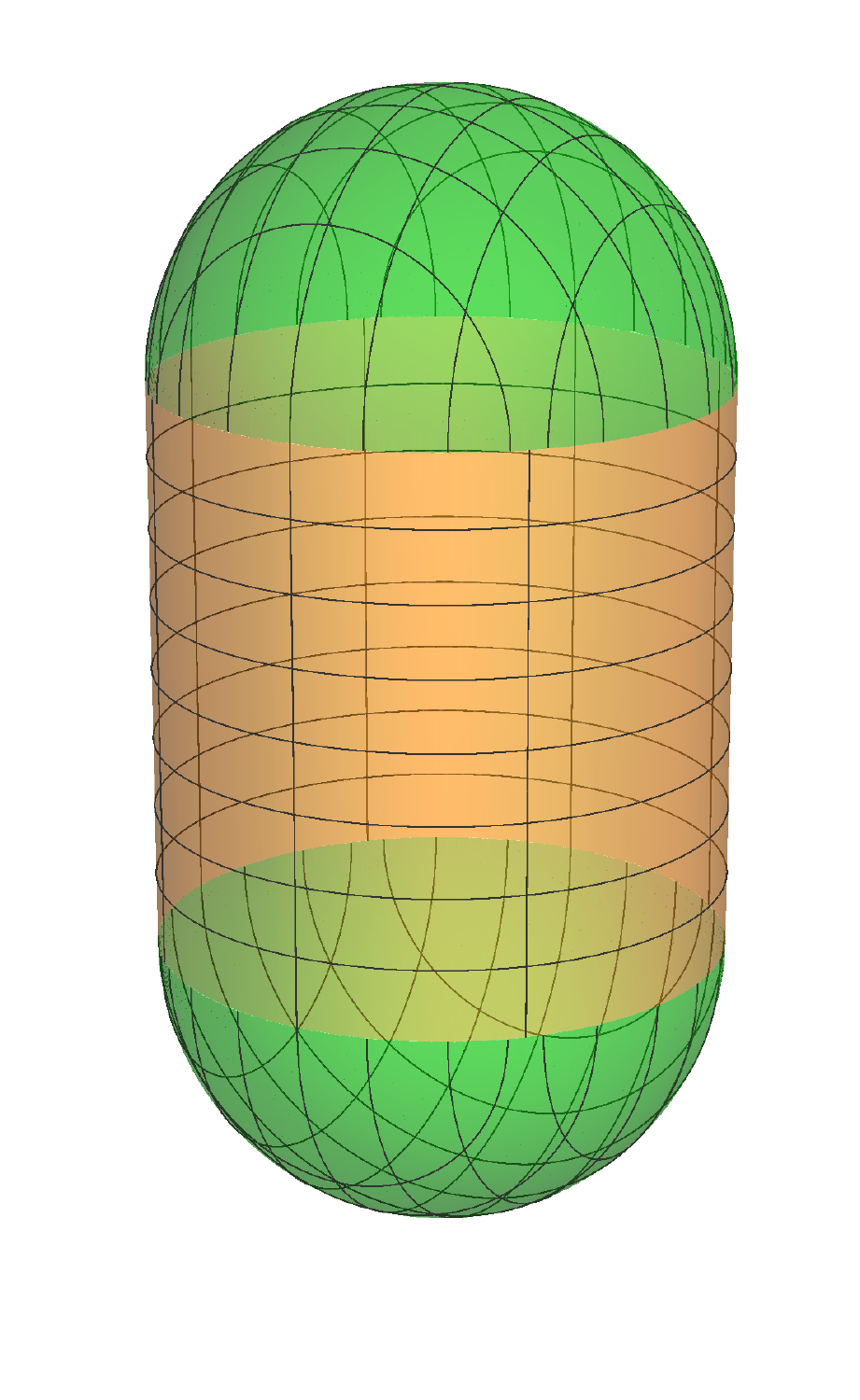}\hspace{2cm}\includegraphics[width=.25\textwidth]{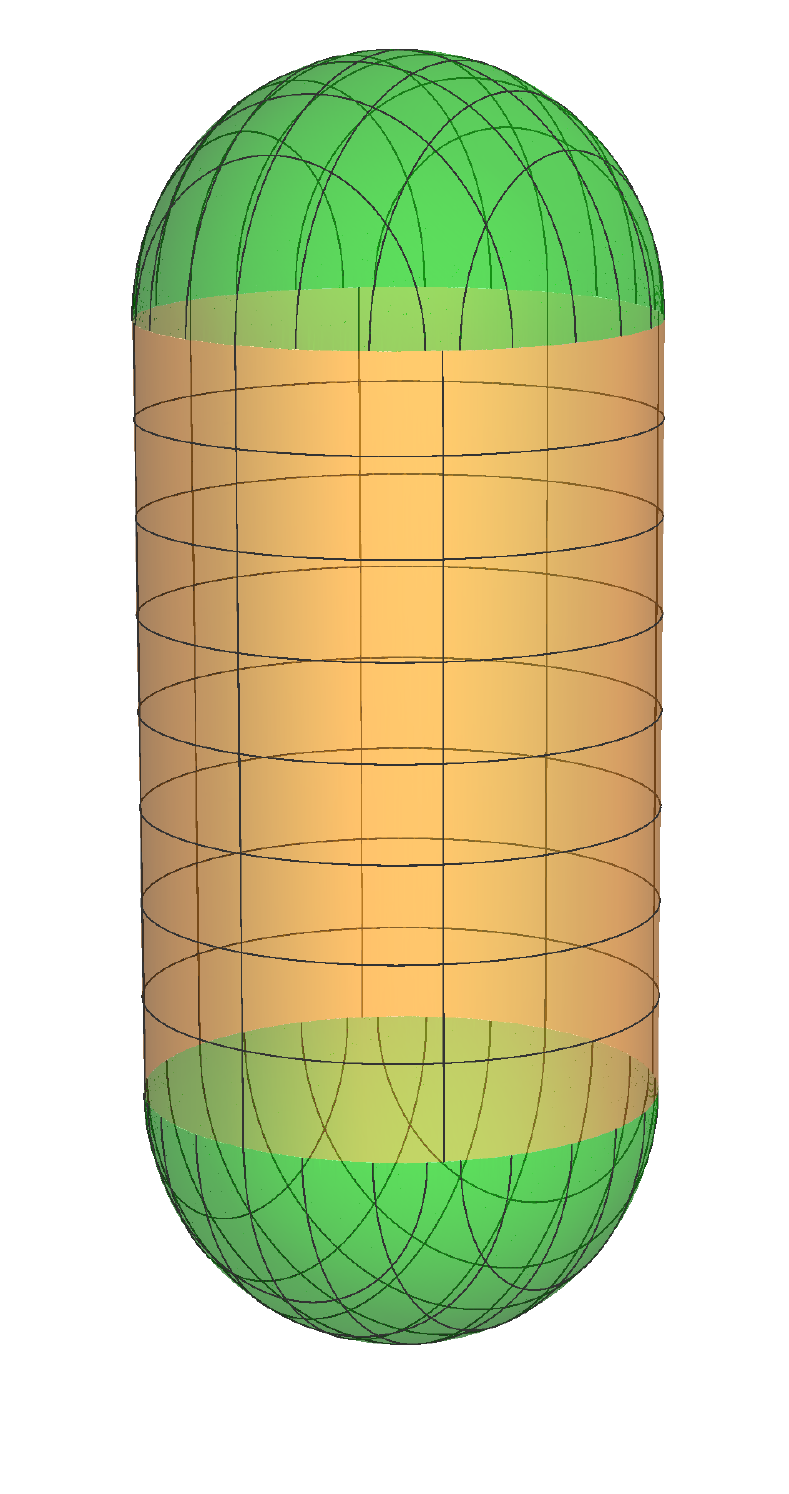}
\end{center}
\caption{Two rotational, not strictly convex, $\cH$-spheres in $\R^3$ for an adequate rotationally symmetric $\cH\in C^0(\S^2)$.}
\label{esfnoconvex}
\end{figure}

\section{$\mathfrak{h}$-bowls and $\mathfrak{h}$-catenoids in $\R^{n+1}$}\label{sec:bowls}

In this section we show examples of properly embedded rotational $\cH$-hypersurfaces in $\R^{n+1}$, for very general choices of the radially symmetric function $\cH\in C^1(\S^n)$. As before, we let $\mathfrak{h}\in C^1([-1,1])$ be given in terms of $\cH$ by \eqref{presim0}. First, we will construct examples of entire, strictly convex, $\cH$-graphs. In analogy with the theory of self-translating solitons of the mean curvature flow, we will call them \emph{$\mathfrak{h}$-bowls.} 

\begin{pro}\label{hbowls}
Let $\ch\in C^1([-1,1])$ be given by \eqref{presim0} in terms of $\cH$, and suppose that there exists $y_0\in [0,1]$ (resp. $y_0\in [-1,0]$) such that $\ch(y_0)=0$. Then there exists an upwards-oriented (resp. downwards-oriented) entire rotational $\cH$-graph $\Sigma$ in $\R^{n+1}$. Moreover, $\Sigma$ is either a horizontal hyperplane, or a strictly convex graph.
\end{pro}

\begin{proof}
If $\ch(1)=0$ (resp. $\ch(-1)=0$), $\Sigma$ can be chosen to be an upwards-oriented (resp. downwards-oriented) horizontal hyperplane in $\R^{n+1}$, and the result is trivial.

Consider next the case that $\ch(1)>0$, and let $y_0\in [0,1]$ be the largest value of $y$ such that $\ch(y)=0$. 
Since $\ch(y)>0$ on $(y_0,1]$ and $\ch(y_0)=0$, the horizontal graph $\Gamma:=\Theta_1 \cap \{x=\Gamma_1(y)\}$ defined by \eqref{graga} has a connected component given by the restriction of $\Gamma_1 (y)$ to the interval $(y_0,1]$, and satisfies $\Gamma_1(1)=0$ and $\Gamma_1(y_0)\to \8$.

Define now $\Lambda\subset\Theta_1$ by $\Lambda=\{(x,y)\in\Theta_1:  y>y_0\}$, and let $\Lambda^+:=\{(x,y)\in\Lambda;\ x>\Gamma_1(y)\}$  and $\Lambda^-:=\{(x,y)\in\Lambda;\ x<\Gamma_1(y)\}$. These components $\Lambda^+$ and $\Lambda^-$ are the only connected components of $\Lambda\setminus \Gamma$, and they have $\Gamma$ as their common boundary. Moreover, $\Lambda^+,\Lambda^-$ are monotonicity regions of $\Theta_1$, and by Lemma \ref{profas} each orbit $y=y(x)$ in $\Lambda^+$ (resp. $\Lambda^-$) satisfies that $y'(x)<0$ (resp. $y'(x)>0$); see Figure \ref{fig:dosff}.

\begin{figure}[h]
\begin{center}
\includegraphics[width=.55\textwidth]{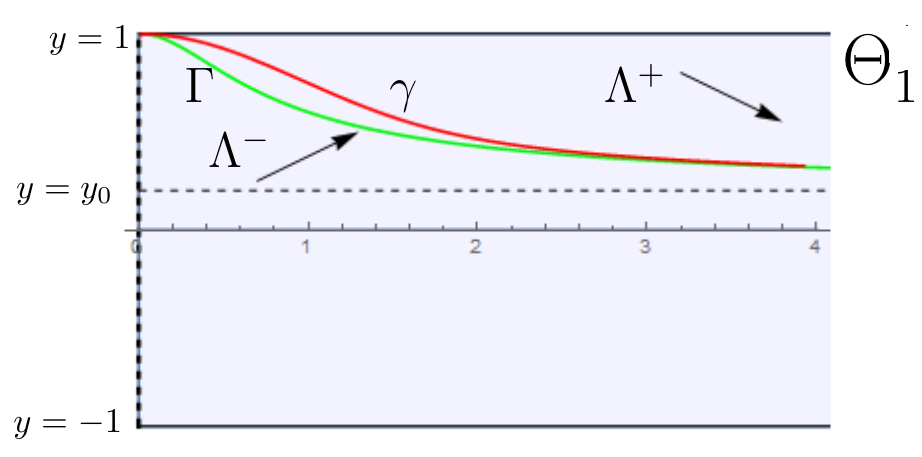}
\end{center}
\caption{The phase space $\Theta_1$, showing the monotonicity directions of $\Lambda^+$, $\Lambda^-$.}
\label{fig:dosff}
\end{figure}

Let now $\Sigma$ be the upwards-oriented rotational $\cH$-graph in $\R^{n+1}$ constructed in Lemma \ref{canoex}, and let $\gamma$ denote the orbit in $\Theta_1$ associated to the profile curve $\alfa(s)=(x(s),z(s))$ of $\Sigma$, which has an endpoint at $(0,1)$; see Corollary \ref{ejefase}. By the monotonicity properties explained above, $\gamma$ lies in $\Lambda^+$ for points near $(0,1)$. By the same monotonicity properties, and using item 4 of Lemma \ref{profas}, we can conclude that  $\gamma$ is globally contained in $\Lambda^+$. Thus, by its monotonocity and properness, $\gamma$ can be seen as a graph $y=r(x)$, where $r\in C^1([0,\8))$ satisfies $r(0)=1$, $r(x)>y_0$ and $r'(x)<0$ for all $x>0$.

This implies that $\Sigma$ is an entire rotational graph in $\R^{n+1}$. Since $\gamma$ is totally contained in $\Theta_{\ep}$ for $\ep=1$, we have by \eqref{signk} that the principal curvatures $\kappa_2,\dots, \kappa_n$ of $\Sigma$ are everywhere positive. Moreover, since $\gamma$ does not leave the monotonicity region $\Lambda^+$, we conclude from \eqref{1ordersys} and \eqref{signk} that $\kappa_1$ is also everywhere positive. Thus, $\Sigma$ is a strictly convex entire graph. 

This concludes the proof in the case that $\ch(1)>0$ and $y_0\geq 0$. A similar argument works in the case that $\ch(-1)>0$ and $y_0\leq 0$. The remaining two cases, namely $\ch(1)<0$, $y_0\geq 0$ and $\ch(-1)<0$, $y_0\leq 0$, are reduced to the previous ones by a change of orientation. This completes the proof of Proposition \ref{hbowls}.
\end{proof}

\begin{figure}[h]
\begin{center}
\includegraphics[width=.3\textwidth,valign=c]{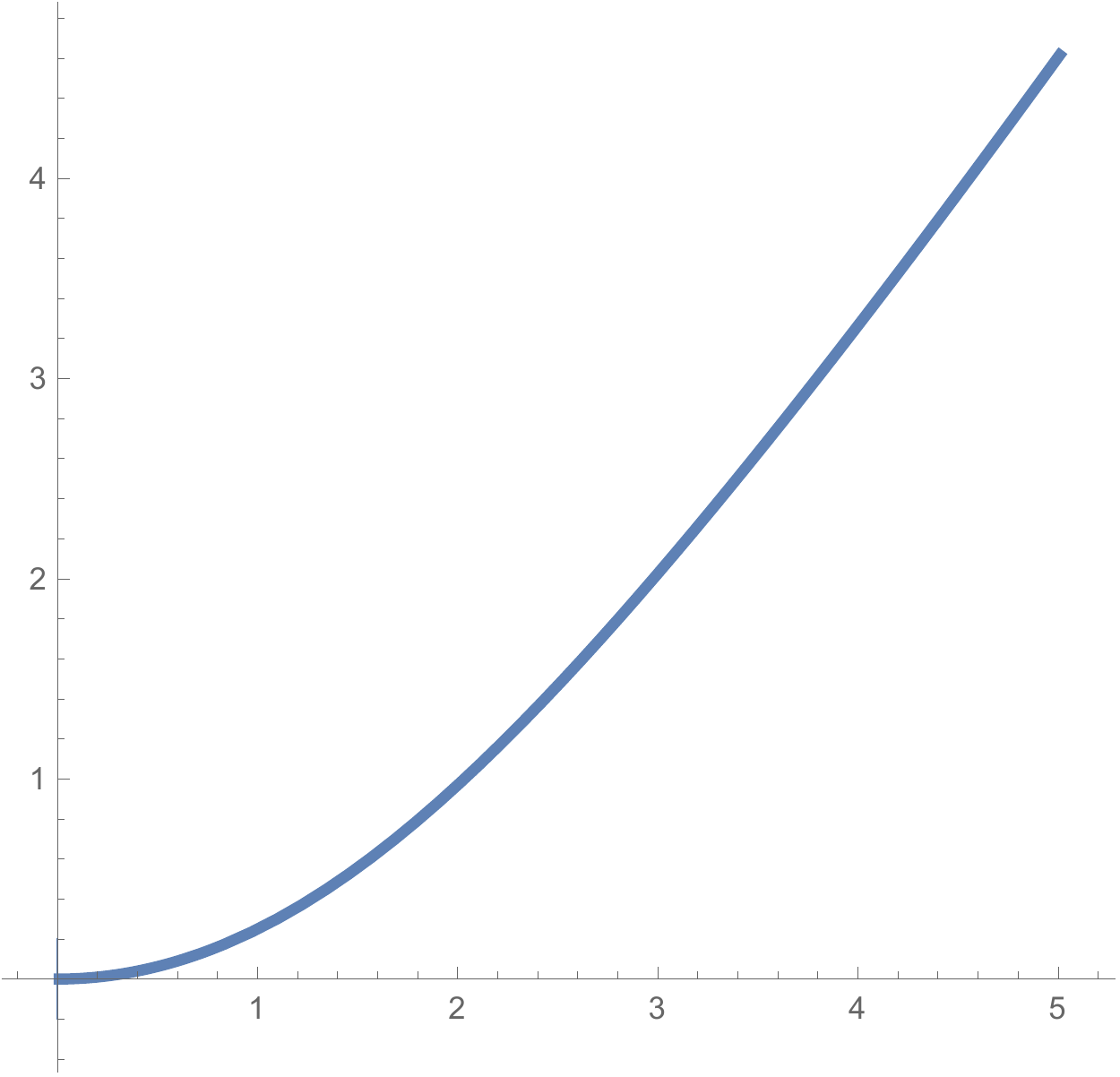}\hspace{1cm}
\includegraphics[width=.55\textwidth,valign=c]{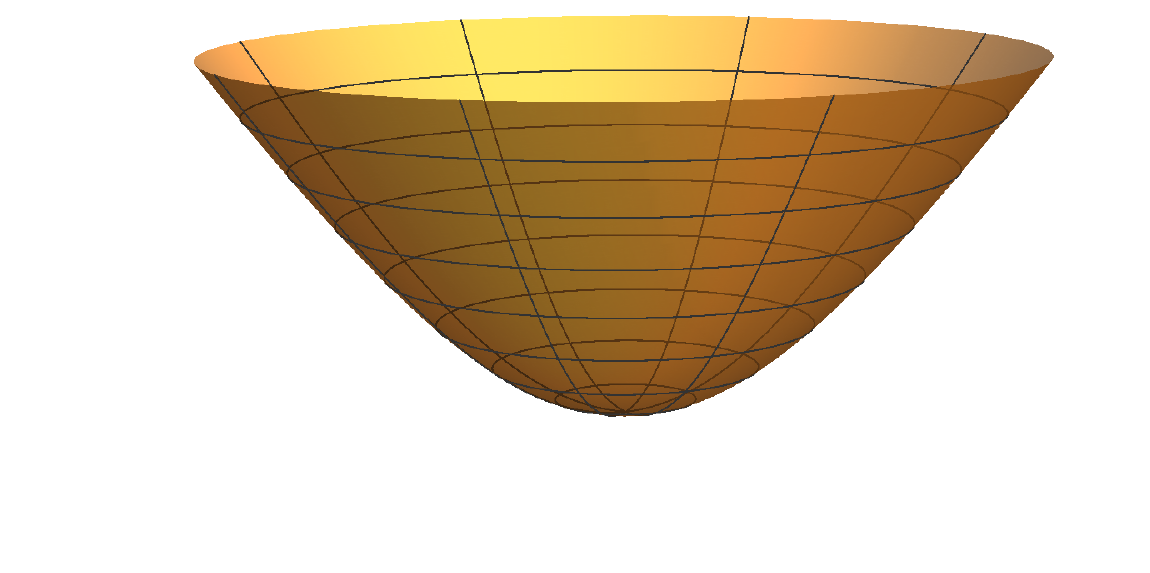}\quad 
\end{center}
\caption{Profile curve and graphic of the $\ch$-bowl in $\R^3$ for the choice $\ch(y)=y-1/2$.}
\label{fig:hbowl}
\end{figure}

The next result shows that there exist \emph{catenoid-type} rotational $\cH$-hypersurfaces for a large class of rotationally invariant choices of $\cH\in C^1(\S^n)$. 

\begin{pro}\label{hcats}
Let $\ch\in C^1([-1,1])$, $\ch\leq 0$, be given by \eqref{presim0} in terms of $\cH$, and suppose that $\ch(\pm 1)=0$.  Then, there exists a one-parameter family of properly embedded rotational $\cH$-hypersurfaces in $\R^{n+1}$ of strictly negative Gauss-Kronecker curvature at every point, and diffeomorphic to $\S^{n-1}\times \R$. Each of them is a bi-graph over $\R^n-\mathbb{B}_n(r_0)$, where $\mathbb{B}_n(r_0)=\{x\in \R^n : |x|<r_0\}$, for some $r_0>0$.
\end{pro}
\begin{proof}
Let $\Sigma$ be the rotational $\cH$-hypersurface in $\R^{n+1}$ generated by a unit speed curve $\alpha(s)=(x(s),z(s))$ that satisfies the initial conditions $x(0)=r_0>0$, $z(0)=0$ and $z'(0)=1$ for some $r_0$; i.e. $\alfa(s)$ corresponds (up to a vertical translation) to the unique solution of \eqref{1ordersys} with initial conditions $x(0)=r_0$, $y(0)=0$. Then, the orbit $\gamma$ of \eqref{1ordersys} associated to $\alfa(s)$ passes through $(r_0,0)$ and belongs to the phase space $\Theta_1$ around that point; i.e. $\varepsilon=1$ in \eqref{1ordersys}.

Observe that, since $\ch\leq 0$, the curve $\Gamma_1:=\Theta_1\cap \{x=\Gamma_1(y)\}$ with $\Gamma_1(y)$ given by \eqref{graga} does not exist (i.e. $\Gamma_1$ is empty). Thus, there are two monotony regions in $\Theta_1$, given by $\Lambda^+:=\Theta_1\cap \{y>0\}$ and $\Lambda^-=\Theta_1\cap \{y<0\}$. Any orbit $y=y(x)$ in $\Lambda^+$ (resp. $\Lambda^-$) satisfies that $y'(x)>0$ (resp. $y'(x)<0$). We should also note that, by the condition $\ch(\pm 1)=0$, and since $\ch\in C^1$, no orbit in $\Theta_1$ can have a limit point of the form $(x,\pm 1)$ for some $x>0$, since this would contradict uniqueness of the solution to the Cauchy problem for \eqref{1ordersys}; observe for this that $(x(s),y(s))=(s,\pm 1)$ is a solution to \eqref{1ordersys}, and that as $\ch\in C^1$ with $\ch(\pm 1)=0$,
the right-hand side of \eqref{1ordersys} is $C^1$ along $y=\pm 1$. 

Taking these properties into account, it is easy to deduce that the orbit $\gamma$ is given as a horizontal graph $x=r(y)$ for some $r\in C^1([a,b])$ with $a<0<b$, and so that $r(0)=r_0$, $r'(y)>0$ (resp. $r'(y)<0$) for all $y\in (0,b)$ (resp. for all $y\in (a,0)$), and $r(y)\to \8$ as $y\to \{a,b\}$. As a matter of fact, using \eqref{yfuncx}, it is easy to show that this is possible in our conditions only if $a=-1$, $b=1$. Thus, $\Sigma$ is a bi-graph in $\R^{n+1}$ over $\Omega:=\R^n-\mathbb{B}_n(r_0)$, with the topology of $\S^{n-1}\times \R$. That is, $\Sigma=\Sigma_1\cup \Sigma_2$ where both $\Sigma_i$ are graphs over $\Omega$ with $\parc \Sigma_i=\parc \Omega$, and $\Sigma_i$ meets the $x_{n+1}=0$ hyperplane orthogonally along $\parc \Sigma_i$.

By \eqref{1ordersys} we get $y'(s)>0$ for all $s$. So, by \eqref{signk}, we have that at every $p\in \Sigma$ the relations $\kappa_1<0$ and $\kappa_2=\dots = \kappa_n>0$ hold. In particular, the Gauss-Kronecker curvature of $\Sigma$ is negative at every point. This completes the proof.
\end{proof}

\begin{figure}[h]
\begin{center}
\includegraphics[width=.4\textwidth,valign=c]{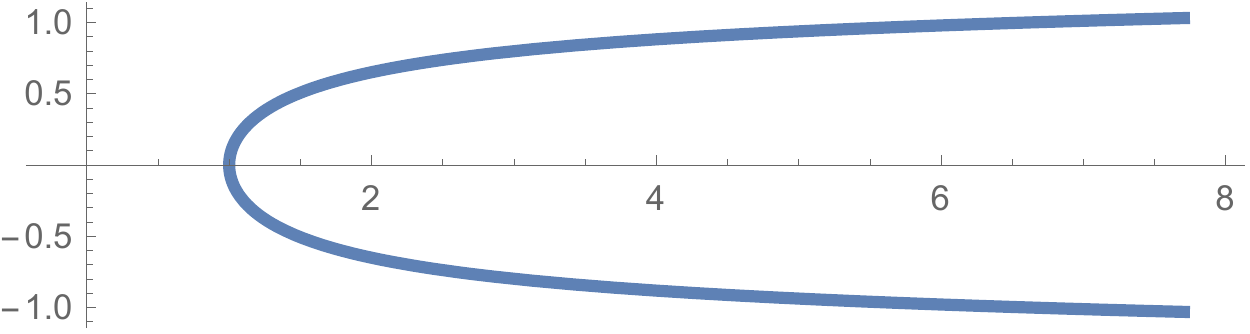}\quad 
\includegraphics[width=.55\textwidth,valign=c]{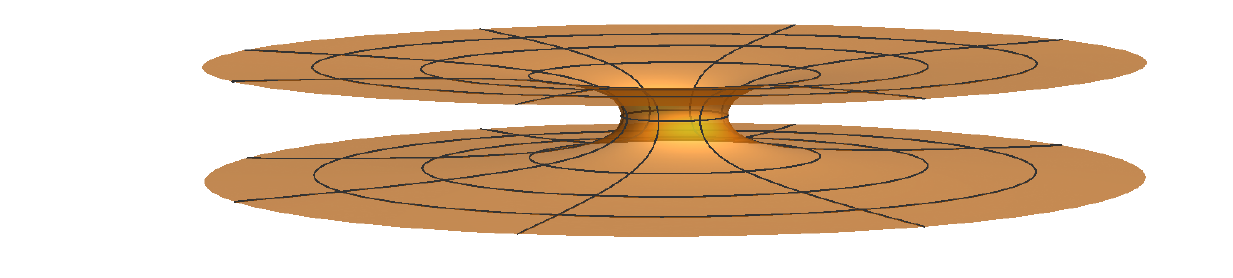}\quad 
\end{center}
\caption{Profile curve and graphic of the $\ch$-catenoid in $\R^3$, for $\ch(y)=y^2-1$.}
\label{fig:hcat}
\end{figure}

We remark that these examples resemble indeed the usual \emph{catenoids} in minimal hypersurface theory, and that Proposition \ref{hcats} recovers them for the choice $\cH=0$. However, while minimal catenoids can be given an explicit parametrization, this is not the case for the $\ch$-catenoids constructed in Proposition \ref{hcats}. Moreover, these $\ch$-catenoids can have different asymptotic behaviors, i.e. different growths at infinity, depending on the choice of the prescribed mean curvature function $\ch$. Understanding this asymptotic behavior can lead to proving \emph{half-space theorems} for properly immersed $\cH$-hypersurfaces in $\R^{n+1}$ for some rather general choices of rotationally symmetric functions $\cH$, by using the classical ideas of Hoffman and Meeks in \cite{HM} about properly immersed minimal surfaces in $\R^3$.  See e.g. \cite{B3} for some results in this direction.

\section{Classification: a Delaunay-type theorem}\label{sec:rot2}

Given $H>0$, a classical theorem by Delaunay shows that there are exactly four types of rotational hypersurfaces in $\R^{n+1}$ with constant mean curvature $H$, up to isometry: round spheres $\S^n(1/H)$, right circular cylinders $\S^{n-1}(r_0)\times \R$ of radius $r_0=(n-1)/(nH)$, a one-parameter family of properly embedded \emph{unduloids}, and a one-parameter family of non-embedded \emph{nodoids}. Both unduloids and nodoids are invariant under the discrete $\Z$-group generated by some vertical translation in $\R^{n+1}$.

In this section we will extend Delaunay's classification, and show that a similar description holds for rotational $\cH$-hypersurfaces in $\R^{n+1}$, in the case that $\cH\in C^1(\S^n)$ is positive, rotationally invariant and \emph{even}, i.e $\cH(x)=\cH(-x)$ for all $x\in \S^n$. In terms of the function $\ch$ associated to $\cH$ by \eqref{presim}, this means that $\ch(-y)=\ch(y)>0$ for all $y\in [-1,1]$. 

We should point out that, in contrast with the classical CMC case, for a prescribed function $\cH$ in the conditions above, the Delaunay hypersurfaces with prescribed $\cH$ cannot be given a general explicit description by integral formulas. We should also point out that this Delaunay-type classification is not true anymore if the prescribed mean curvature $\cH>0$ is not an even function on $\S^n$. In Section \ref{sec:nocmc} we will provide examples that support this claim.

\begin{teo}\label{dela}
Let $\ch\in C^1([-1,1])$ be given by \eqref{presim0} in terms of $\cH$, and suppose that $\ch(y)=\ch(-y)>0$ for all $y$. Then, the following list exhausts, up to vertical translations, all existing rotational $\cH$-hypersurfaces in $\R^{n+1}$ around the $x_{n+1}$-axis:
\begin{enumerate}
\item The right circular cylinder $C_{\cH}:=\S^{n-1}(r_0)\times \R$, where $r_0=(n-1)/(n\ch(0))$.

\item A strictly convex rotational $\cH$-sphere $S_{\cH}\subset \R^{n+1}$.

\item A one-parameter family of properly embedded \emph{$\mathfrak{h}$-unduloids}, $U_\mathfrak{h}$. 
\item A one-parameter family of properly immersed, non-embeded, \emph{$\mathfrak{h}$-nodoids}, $N_\mathfrak{h}$.
\end{enumerate}
Moreover, any $\ch$-unduloid or $\ch$-nodoid is invariant by a vertical translation in $\R^{n+1}$, diffeomorphic to $\S^{n-1}\times \R$, and lies in a tubular neighborhood of the $x_{n+1}$-axis in $\R^{n+1}$. The Gauss map image of each $\ch$-unduloid omits open neighborhoods of the north and south poles in $\S^n$. The Gauss map image of each $\ch$-nodoid is $\S^n$.
\end{teo}

\begin{proof}
Let $(x(s),y(s))$ be any solution to \eqref{1ordersys}. Since $\ch(y)=\ch(-y)>0$, it follows that $(x(-s),-y(-s))$ is also a solution to \eqref{1ordersys}. Geometrically, this means that any orbit of the phase space $\Theta_{\ep}$, $\ep=\pm 1$, is symmetric with respect to the $y=0$ axis. The curve $\Gamma_1$ in $\Theta_1$ given by \eqref{graga} together with $y=0$ divides the phase space $\Theta_1$ into four monotonicity regions $\Lambda_1,\dots, \Lambda_4$, all of them meeting at the equilibrium $e_0$ in \eqref{equil}. The curve $\Gamma_{-1}$ in $\Theta_{-1}$ does not exist, and so $\Theta_{-1}$ has only two monotony regions: $\Lambda^+= \Theta_{-1}\cap \{y>0\}$ and $\Lambda^-=\Theta_{-1}\cap \{y<0\}$. See Figure \ref{fig:2fa}.

\begin{figure}[h]
\centering
\includegraphics[width=0.45\textwidth]{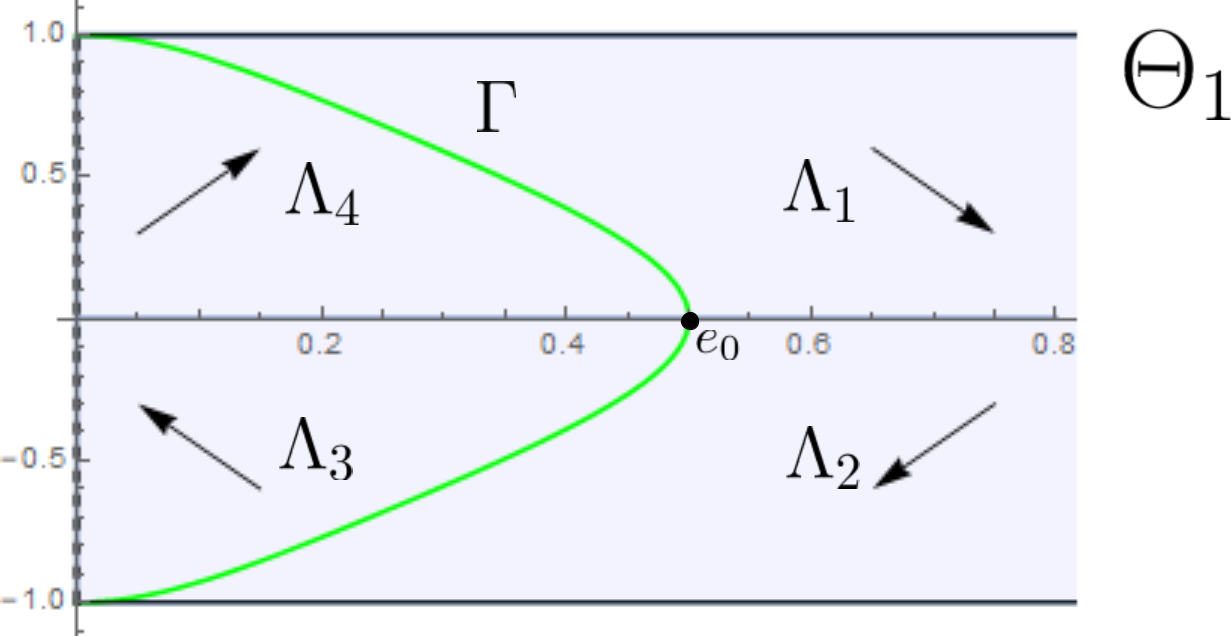}\hspace{1cm}
\includegraphics[width=0.477\textwidth]{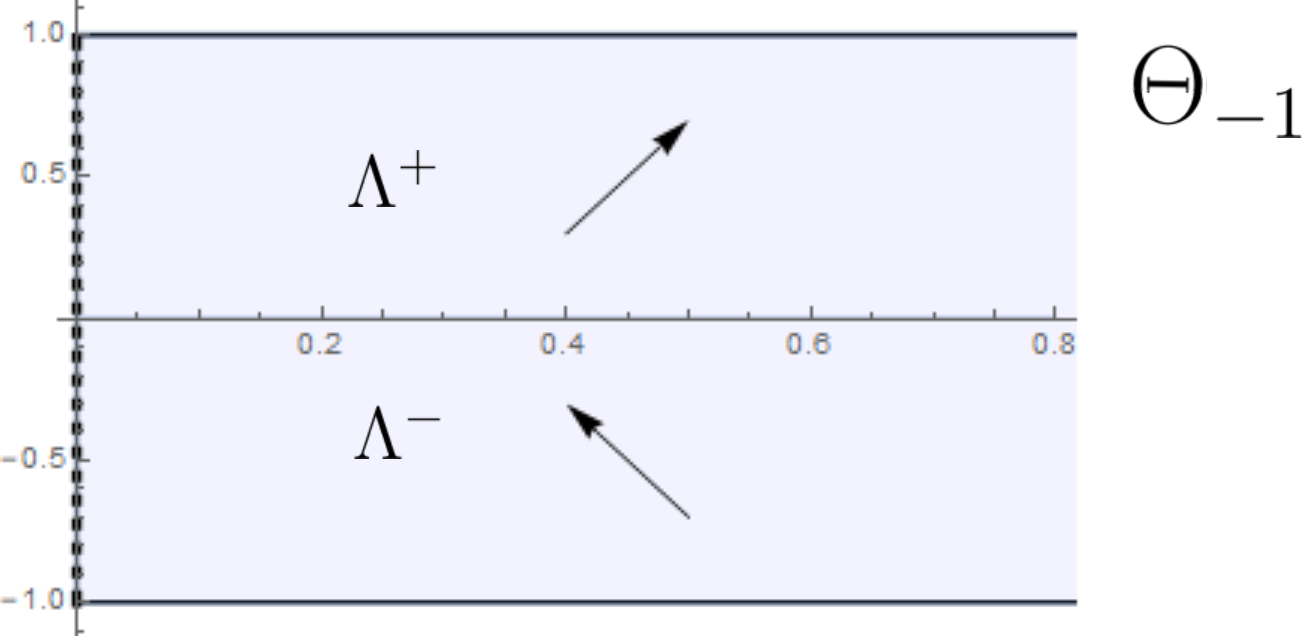}
\caption{The phase spaces $\Theta_1$ and $\Theta_{-1}$ for $\ch(y)=\ch(-y)>0$.}
\label{fig:2fa}
\end{figure}

Let us prove next that \emph{no orbit $\gamma$ in $\Theta_1$ or $\Theta_{-1}$ can have a limit point of the form $(0,y)$, $|y|<1$}. Indeed, assume that such an orbit $\gamma$ exists, and let $\alfa(s)=(x(s),z(s))$ denote the profile curve of its corresponding rotational $\cH$-hypersurface $\Sigma$. Then, $(x(s_n),x'(s_n))\to (0,y)$ for a sequence of values $s_n$, and in particular $\alfa(s)$ approaches the rotation axis in a non-orthogonal way (since $|y|\neq 1$). So, by the monotonicity properties of the phase space, we see that a piece of $\Sigma$ is a graph $x_{n+1}=u(x_1,\dots, x_n)$ on a punctured ball $\Omega-{\bf \{0\}}$ in $\R^{n+1}$. Moreover, the mean curvature function of $\Sigma$, viewed as a function $H(x_1,\dots, x_n)$ on $\Omega-{\bf \{0\}}$, extends continuously to the puncture, with value $\ch(y)$. Hence, it is known that the graph $\Sigma$ extends smoothly to the ball $\Omega$, see e.g. \cite{LRo}. In particular, the unit normal at the puncture is vertical. This is a contradiction with $|y|<1$.

The description of the orbits in $\Theta_{-1}$ follows then easily from the fact above, and the previous monotonicity properties. Any such orbit is given by a horizontal $C^1$ graph $x=g(y)$, with $g(y)=g(-y)>0$ for every $y\in (-\delta,\delta)$ for some $\delta\in (0,1]$, and such that $g$ restricted to $[0,\delta)$ is strictly increasing. In the case that $g(y)\to \8$ as $y\to \delta$ for some orbit, 
the rotational $\cH$-hypersurface in $\R^{n+1}$ described by that orbit would be a symmetric bi-graph over the exterior of an open ball in $\R^n$. This is impossible, since $\cH\geq H_0>0$ for some $H_0>0$; indeed, it is well known that there do not exist graphs in $\R^{n+1}$ over the exterior of a ball, with mean curvature bounded from below by a positive constant. 
Hence, $\delta=1$ and any orbit in $\Theta_{-1}$ has two limit endpoints of the form $(x_0,\pm 1)$ for some $x_0>0$. The resulting $\cH$-hypersurface $\Sigma_{-1}$ in $\R^{n+1}$ associated to any such orbit is then a compact (with boundary) symmetric bi-graph over some domain $\Omega\subset \R^n$ of the form $\{x \in \R^n: \alfa\leq |x|\leq x_0\}$, and its boundary is given by 
 \begin{equation}\label{bosi}
\parc \Sigma_{-1} =( \S^{n-1} (x_0) \times \{a\} )\cup (\S^{n-1} (x_0) \times \{b\}),
\end{equation}
for some $a<b$. The $z(s)$-coordinate of the profile curve $\alfa(s)$ of $\Sigma_{-1}$ is strictly decreasing, and the unit normal to $\Sigma_{-1}$ along $\parc \Sigma_{-1}\cap \{x_{n+1}=a\}$ (resp. along $\parc \Sigma_{-1}\cap \{x_{n+1}=b\}$) is constant, and equal to $e_{n+1}$ (resp. to $-e_{n+1}$). 

We describe next the orbits in $\Theta_1$ together with their associated $\cH$-hypersurfaces. First, observe that the equilibrium $e_0\in \Theta_1$, given by \eqref{equil}, corresponds to the cylinder $C_{\cH}\subset \R^{n+1}$.

\begin{figure}[h]
\centering
\includegraphics[width=0.45\textwidth,valign=c]{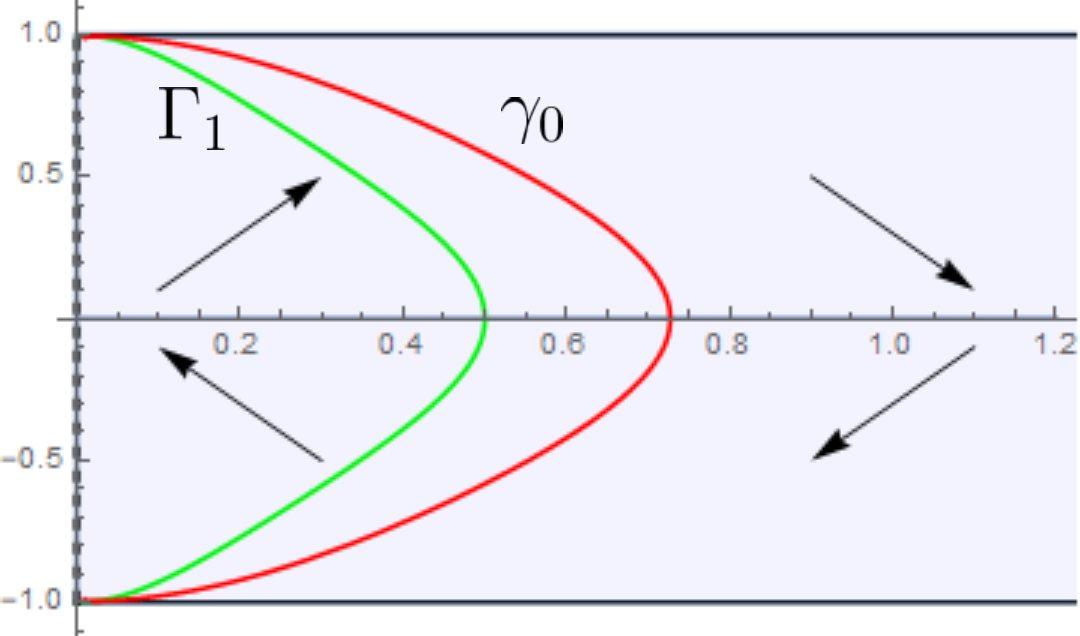}
\hspace{1cm}
\includegraphics[width=0.25\textwidth,valign=c]{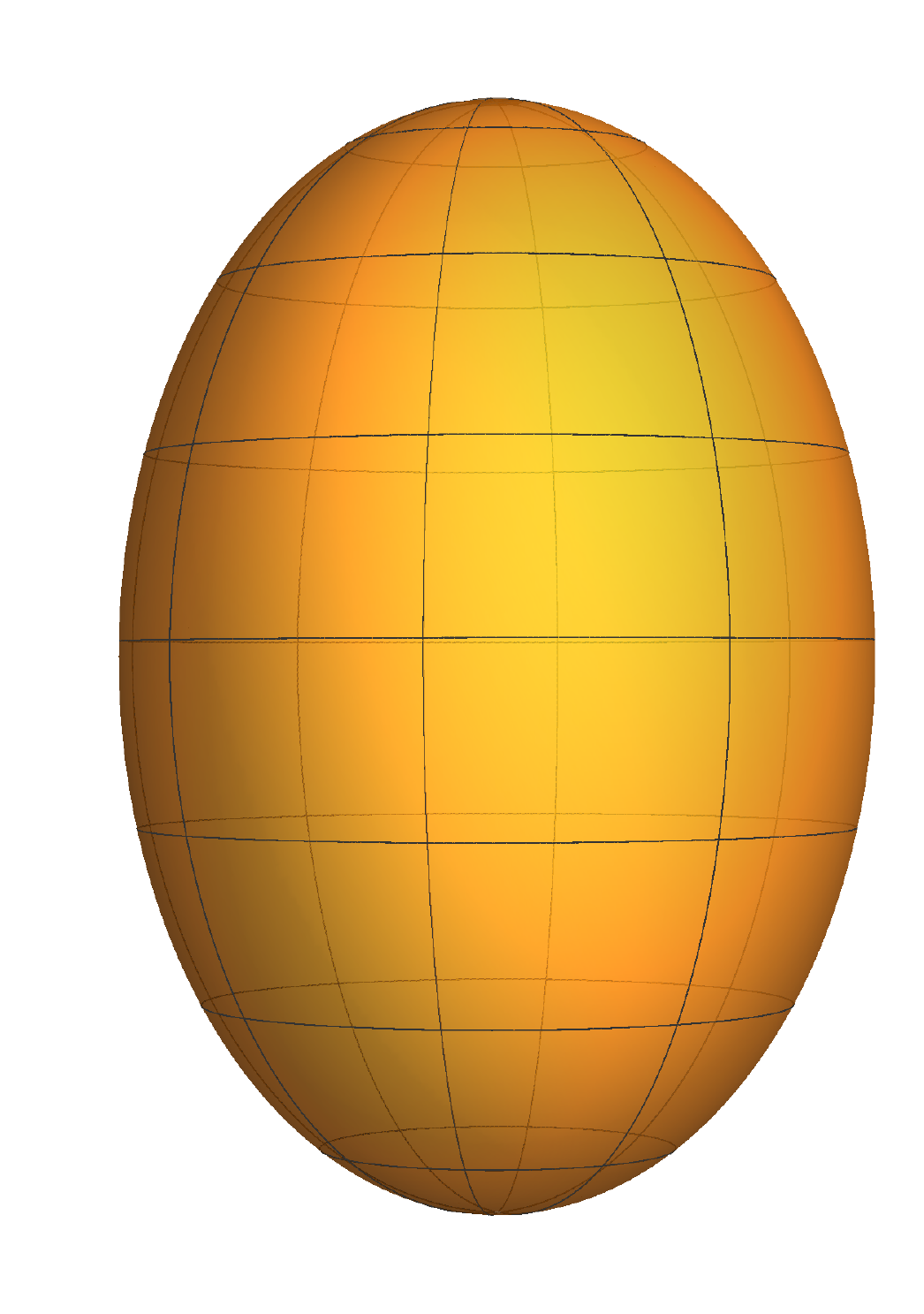}
\caption{Left: Phase space $\Theta_1$ for the choice $\ch(y)=1+y^2$ for surfaces in $\R^3$. The orbit $\gamma_0$ corresponds to the convex $\cH$-sphere $S_{\cH}$ in $\R^3$ (right).} 
\label{fig:gg}
\end{figure}

Let us analyze the structure of the orbits of $\Theta_1$ around $e_0$. Noting that $\ch'(0)=0$ by symmetry, we can check that the linearized system at $e_0$ associated to the nonlinear system \eqref{1ordersys} for $\ep=1$ is
\begin{equation}\label{2ordersys}
\left(\begin{array}{c}
u\\
v
\end{array}\right)'=\left(\begin{array}{cc}
0 & 1\\
- n^2\ch(0)^2/(n-1) & 0
\end{array}\right) \left(\begin{array}{c}
u\\
v
\end{array}\right),
\end{equation}
whose orbits are ellipses around the origin. By classical theory of nonlinear autonomous systems, this means that there are two possible configurations for the space of orbits of \eqref{1ordersys} near $e_0$; either all such orbits are closed curves (a \emph{center} structure), or they spiral around $e_0$. However, this second possibility cannot happen, since all orbits of \eqref{1ordersys} are symmetric with respect to the axis $y=0$, and $e_0$ lies in this axis. In particular, we deduce that \emph{all orbits of $\Theta_1$ stay at a positive distance from the equilibrium $e_0$.
}

Next, observe that by Corollary \ref{ejefase} there exists an orbit $\gamma_0$ in the phase space $\Theta_1$ that has $(0,1)$ as a limit point. Note that $\gamma_0$ lies near $(0,1)$ inside the monotonicity region $\Lambda_1$, and it cannot stay forever in that region, by an argument as the one above using that $\cH\geq H_0>0$. Moreover, $\gamma_0$ stays at a positive distance from $e_0$, by the previous discussion, and it does not intersect $\Gamma_1$. Thus, $\gamma_0$ must intersect the $y=0$ axis at some point $(x_0,0)\neq e_0$, with $x_0>0$. By symmetry of the phase space, we deduce that $\gamma_0$ can be extended to be an orbit in $\Theta_1$ that joins the limit points $(0,1)$ and $(0,-1)$, and that lies in the region $\Lambda_1\cup \Lambda_2\cup \{y=0\}$; see Figure \ref{fig:gg}. It is clear that the rotational $\cH$-hypersurface associated to the orbit $\gamma_0$ is a strictly convex $\cH$-sphere $S_{\cH}$ (see Theorem \ref{teoes}).

Also, observe that $\gamma_0$ divides $\Theta_1$ into two connected components: the one containing the equilibrium $e_0$, which we will denote by $\cW_0$, and the one where $x>0$ is unbounded, which will be denoted  by $\cW_{\8}$. Any orbit of $\Theta_1$ other than $\gamma_0$ lies entirely in one of these two open sets. 

By symmetry and monotonicity, and the fact that $\gamma_0$ is the unique orbit in $\Theta_{1}$ with an endpoint of the form $(0,\pm 1)$, it is clear that any orbit in $\cW_{\8}$ is a symmetric horizontal graph $x=g(y)$, with $g(1)=g(-1)=x_0$ for some $x_0>0$, and such that $g$ is strictly increasing (resp. decreasing) in $(-1,0]$ (resp. in $[0,1)$). Let $\Sigma_1$ denote the rotational $\cH$-hypersurface in $\R^{n+1}$ associated to any such orbit in $\cW_{\8}$, and let $\alfa(s)=(x(s),z(s))$ be its profile curve. Note that $z'>0$ since $\ep=1$. By similar arguments to the ones used for $\Theta_{-1}$, we conclude that $\Sigma_1$ is a compact (with boundary) symmetric bi-graph in $\R^{n+1}$ over some domain in $\R^n$ of the form $\{x \in \R^n: x_0\leq |x|\leq \beta \}$, and  
 \begin{equation}\label{bosi}
\parc \Sigma_1 =( \S^{n-1} (x_0) \times \{c\} )\cup (\S^{n-1} (x_0) \times \{d\}),
\end{equation}
for some $c<d$. The unit normal to $\Sigma_1$ along $\parc \Sigma_1\cap \{x_{n+1}=c\}$ (resp. along $\parc \Sigma_1\cap \{x_{n+1}=d\}$) is $e_{n+1}$ (resp. $-e_{n+1}$). 
\begin{figure}[h]
\centering
\includegraphics[width=0.5\textwidth,valign=c]{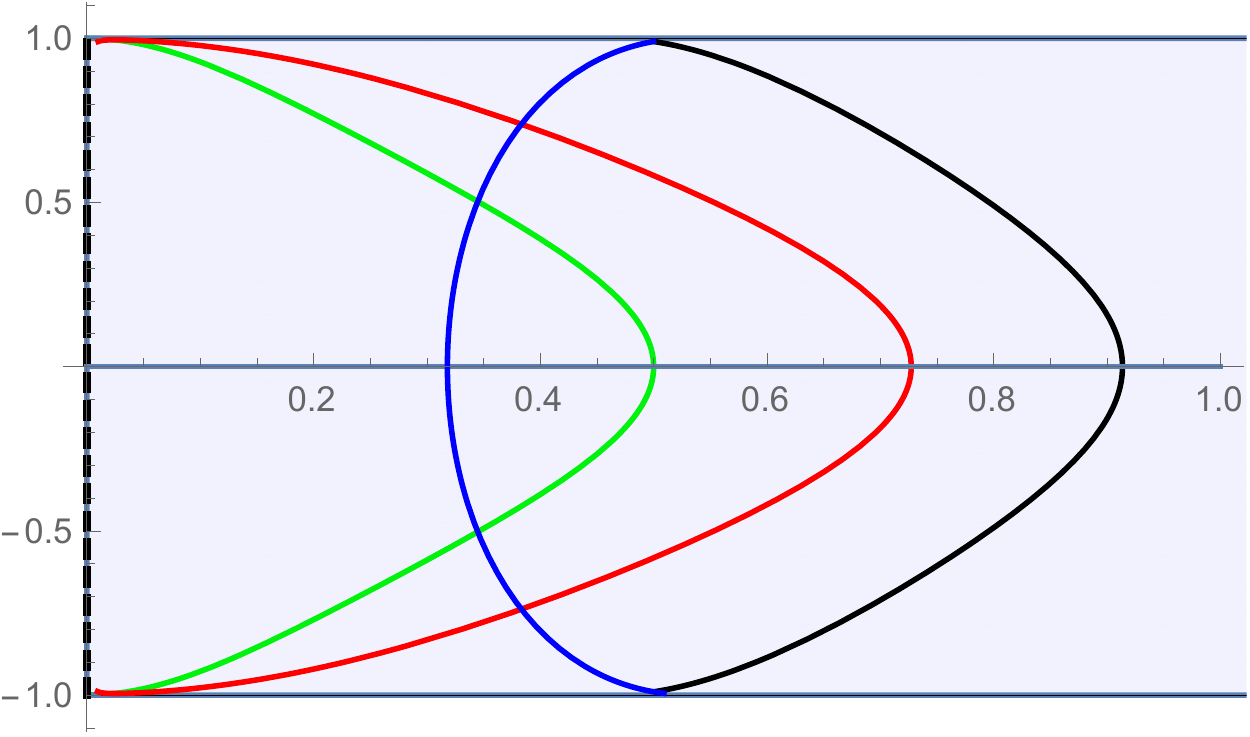}
\hspace{1cm}
\includegraphics[width=0.35\textwidth,valign=c]{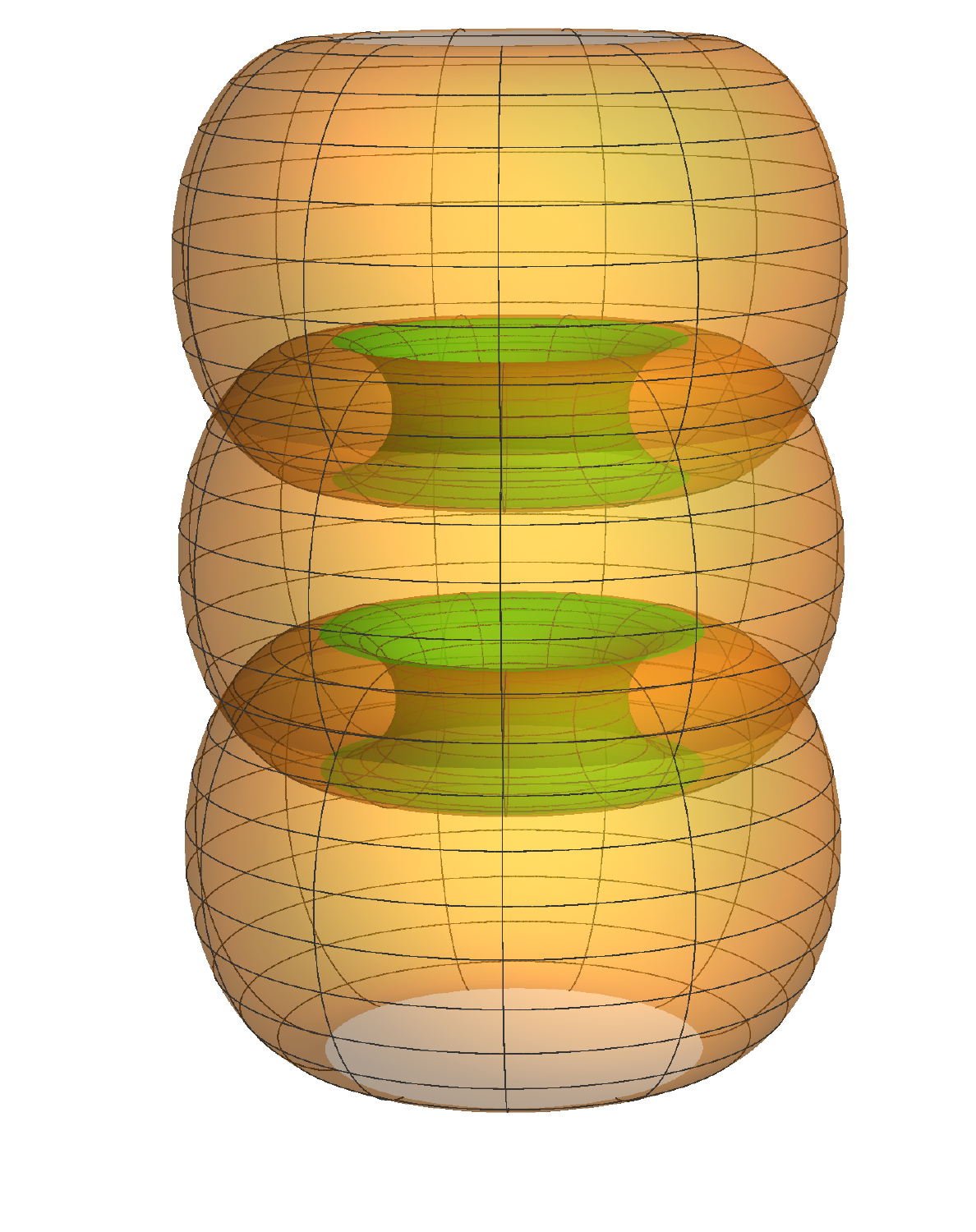}
\caption{Left: Phase space for the choice $\ch(y)=1+y^2$ for surfaces in $\R^3$. The green curve corresponds to $\Gamma_1$, and the red curve to the orbit $\gamma_0$ of the $\cH$-sphere $S_{\cH}$. The black curve is the orbit corresponding to the piece of the $\ch$-nodoid where $z'>0$, and thus lies in $\Theta_1$. The blue curve is the orbit of the $\ch$-nodoid in the $\Theta_{-1}$ phase space (it can intersect other orbits, since it belongs to a different phase space). Right: a picture of a section of the $\ch$-nodoid in $\R^3$.}
\label{fig:nod}
\end{figure}

\begin{figure}[h]
\centering
\includegraphics[width=0.5\textwidth,valign=c]{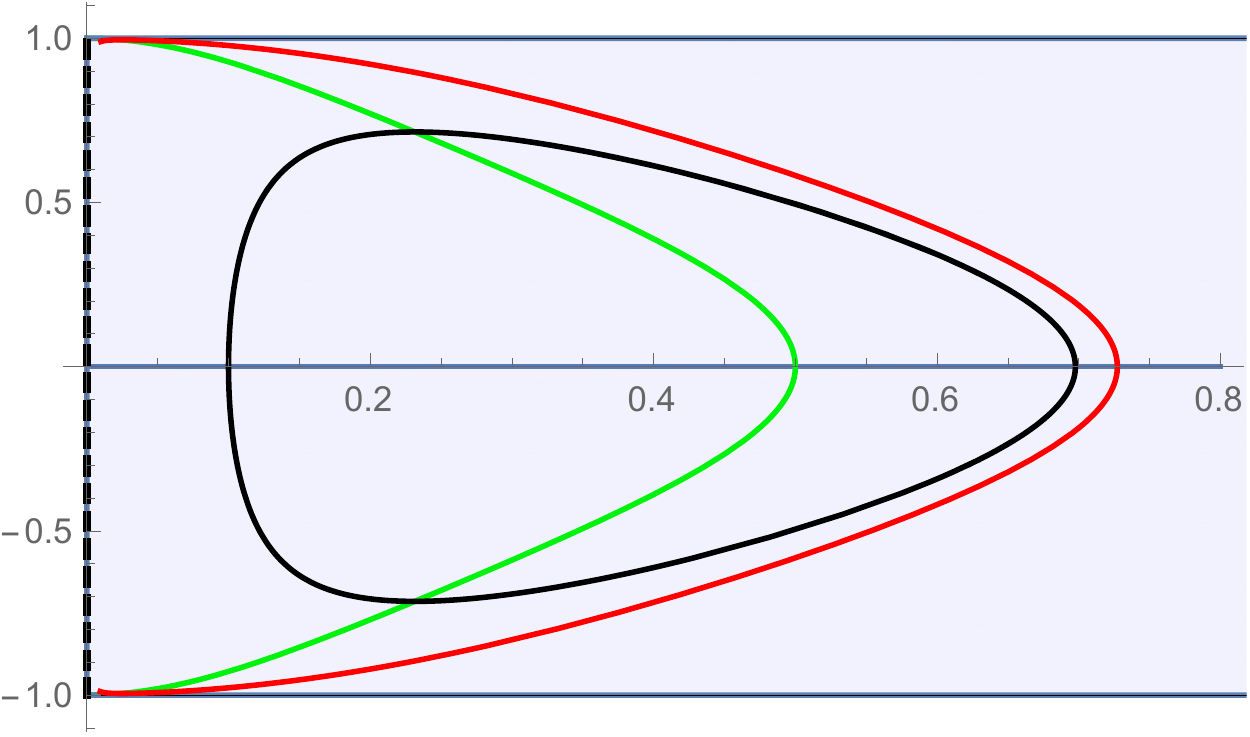}
\hspace{2cm}
\includegraphics[width=0.13\textwidth,valign=c]{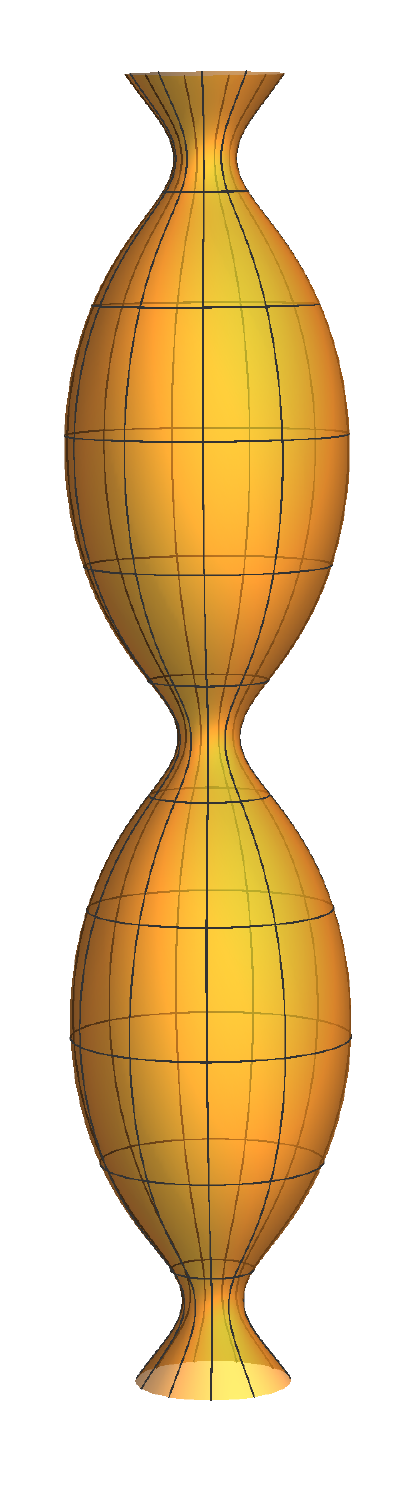}
\caption{Left: Phase space $\Theta_1$ for the choice $\ch(y)=1+y^2$ for surfaces in $\R^3$. The green curve corresponds to $\Gamma_1$, and the red curve to the orbit $\gamma_0$ of the $\cH$-sphere $S_{\cH}$. The black curve is the orbit of an $\ch$-unduloid. Right: a picture of an $\ch$-unduloid in $\R^3$.}
\label{fig:undu}
\end{figure}

Consequently, by uniqueness of the solution to the Cauchy problem for rotational $\cH$-hypersurfaces in $\R^{n+1}$, we can deduce that, given $x_0>0$, the $\cH$-hypersurfaces $\Sigma_{-1}$ and $\Sigma_1$ that we have constructed associated to $x_0$ can be smoothly glued together along any of their boundary components where their unit normals coincide, to form a larger $\cH$-hypersurface. For this, we should note that both $\Sigma_{-1}$ and $\Sigma_1$ are defined up to vertical translations in $\R^{n+1}$, and so we can assume without loss of generality in the previous construction that $a=d$ or that $b=c$ (and hence $\Sigma_1$ and $\Sigma_{-1}$ have the same Cauchy data). There are two cases to consider.

If we have simultaneously $a=d$ and $b=c$, the hypersurface obtained by gluing $\Sigma_1$ with $\Sigma_{-1}$ is a compact rotational $\cH$-hypersurface diffeomorphic to $\S^{n-1}\times \S^1$. But now, note that since $\cH$ is rotationally symmetric, in particular, it is invariant with respect to $n$ linearly independent geodesic reflections of $\S^n$. In these conditions, Proposition 2.8 in \cite{BGM} establishes that the hypersurface must be diffeomophic to $\S^n$, and this contradicts the previous information on its topology.

Consequently, we have $a=d$ and $b\neq c$, or $b=c$ and $a\neq d$. In that way, by iterating the previous process we obtain a proper, non-embedded rotational $\cH$-hypersurface in $\R^{n+1}$ diffeomorphic to $\S^{n-1}\times \R$, invariant by a vertical translation. This proves the existence of the family of $\ch$-nodoids $N_{\ch}$ in $\R^{n+1}$.

To end the proof of Theorem \ref{dela}, we consider an orbit $\gamma$ of $\Theta_1$ that is contained in the region $\cW_0$. Recall that we proved previously that $\gamma$ stays at a positive distance from the equilibrium $e_0$. Also, recall that we showed at the beginning of the proof that $\gamma$ cannot approach a point of the form $(0,y)$ with $y\in (-1,1)$.  So, as $\gamma$ is symmetric with respect to the $y=0$ axis, if we  take into account the monotonocity properties of $\Theta_1$, we deduce that $\gamma$ is a closed curve containing $e_0$ in its inner region. This implies that the profile curve $\alfa(s)=(x(s),z(s))$ of the rotational $\cH$-hypersurface associated to any such orbit satisfies that $s$ is defined for all real values, that $z'(s)>0$ for all $s$, and that $x(s)$ is periodic. These properties imply that $\Sigma$ is an $\ch$-unduloid, with all the properties listed in the statement of the theorem (see Figure \ref{fig:undu}). This concludes the proof.

%
%
%

%
\end{proof}

It is interesting to remark that the family of $\ch$-unduloids for a given $\ch$ in the conditions of Theorem \ref{dela} is a continuous $1$-parameter family. Similarly to what happens in the CMC case, at one extreme of the parameter domain we have a (singular) vertical chain of tangent rotational $\cH$-spheres $S_{\cH}$, and at the other extreme we have the $\cH$-cylinder $C_{\cH}$.

\begin{figure}[h]
\begin{center}
\includegraphics[width=.25\textwidth,valign=c]{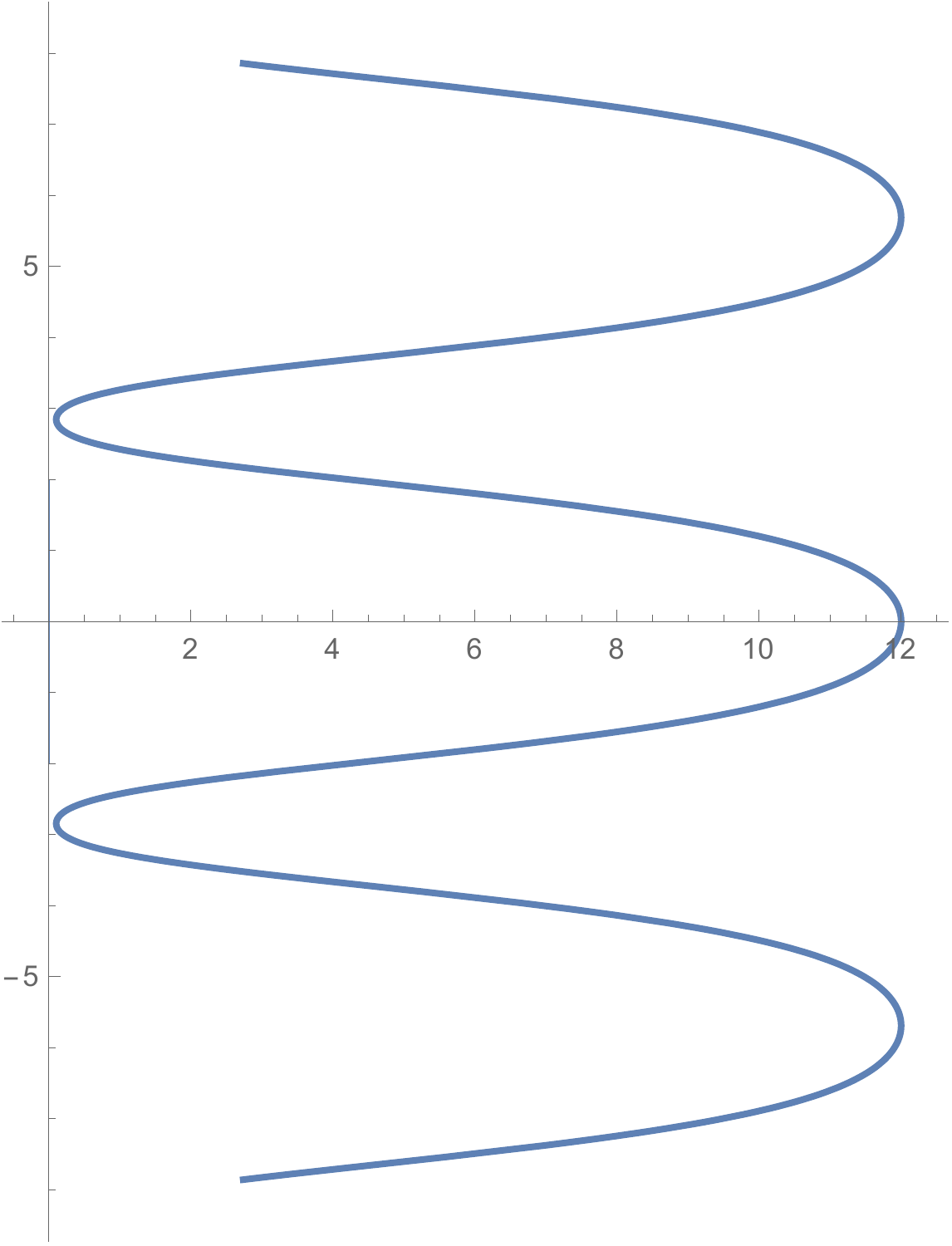}\quad
\hspace{2cm}
\includegraphics[width=.5\textwidth,valign=c]{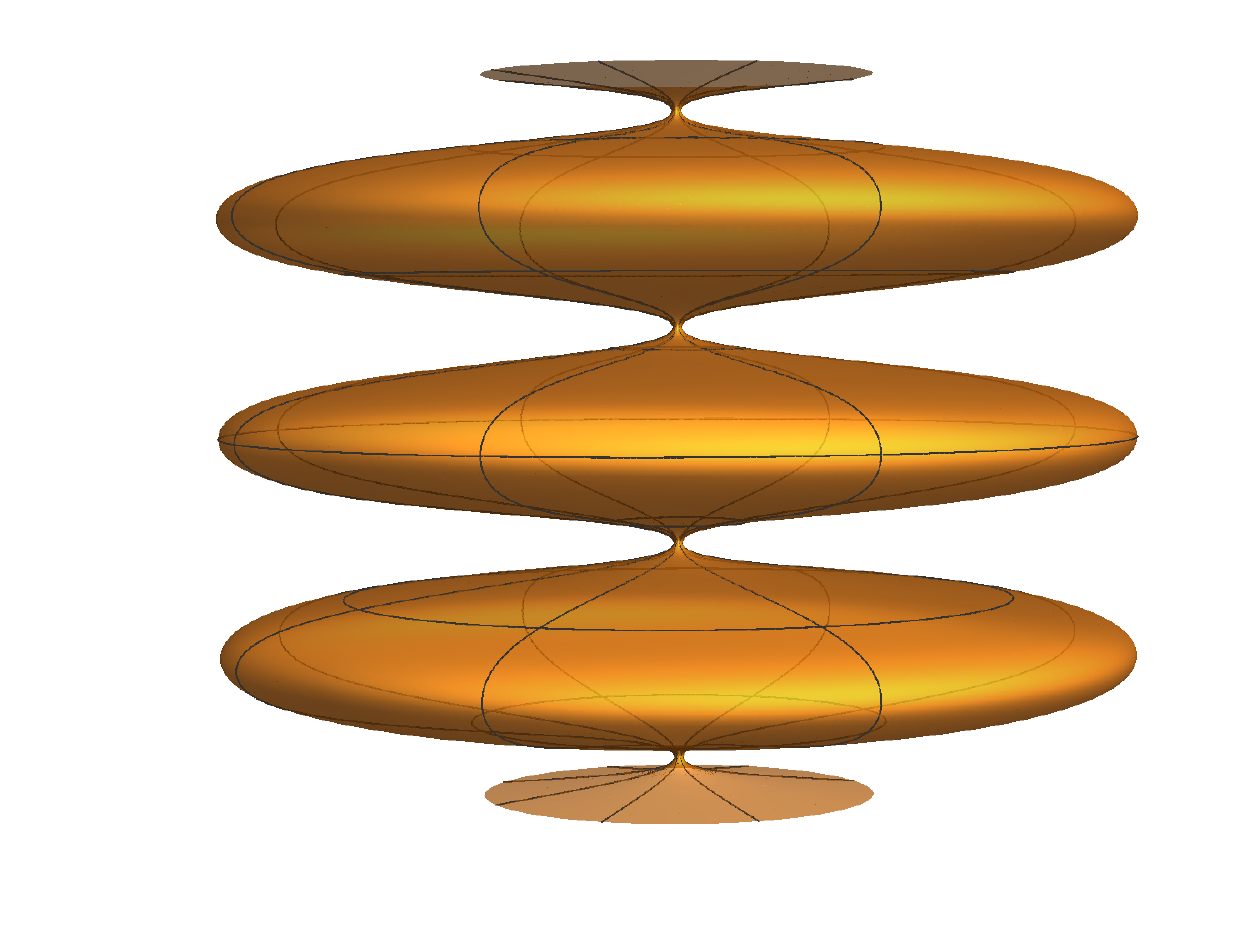}\quad 
\caption{Profile curve and picture of an $\ch$-unduloid in $\R^3$ for $\ch(y)=1-y^2$.}
\label{fig:unh}
\end{center}
\end{figure}

It is also important to stress that the ideas in the proof of Theorem \ref{dela} can be used to classify the rotational $\cH$-hypersurfaces in $\R^{n+1}$ for other types of rotationally symmetric choices of $\cH$.

For instance, assume that $\cH\in C^1(\S^n)$ is rotationally symmetric, and its associated function $\ch\in C^1([-1,1])$ satisfies that $\ch(y)=\ch(-y)>0$ for every $y\in (-1,1)$, and $\ch(\pm 1)=0$. In that case it can be shown using the previous ideas (although details will be skipped here) that there are exactly four types of rotational $\cH$-hypersurfaces in $\R^{n+1}$ with axis $x_{n+1}$: \emph{horizontal hyperplanes, cylinders $C_{\cH}$ of the form $\S^n(r_0)\times \R$, $\ch$-catenoids of non-vanishing mean curvature like the ones of Proposition \ref{hcats} (with opposite orientation), and hypersurfaces of $\ch$-unduloid type} similar to those of Theorem \ref{dela}. Interestingly, in this situation the family of $\ch$-unduloids varies between the cylinder $C_{\cH}$ and a singular family of double covers of horizontal hyperplanes joined along the $x_{n+1}$-axis. See Figure \ref{fig:unh}.

We give next an application of Theorem \ref{dela}. Let us recall that constant mean curvature spheres in $\R^3$ are described by a classical theorem of H. Hopf: \emph{any surface of constant mean curvature surface $H$ immersed in $\R^3$ and diffeomorphic to $\S^2$ is a sphere of radius $1/H$}. In \cite{GM1,GM2,GM3} the last two authors proved several extensions of Hopf's theorem to very general geometric situations. In the context of $\cH$-surfaces in $\R^3$, their result can be phrased as follows:

\begin{teo}[\cite{GM1,GM2}]\label{hopf}
Let $\cH\in C^0(\S^2)$, and assume that there exists a strictly convex $\cH$-sphere $S$ in $\R^3$. Then any $\cH$-surface immersed in $\R^3$ diffeomorphic to $\S^2$ is a translation of $S$.
\end{teo}

So, as a direct application of Theorems \ref{dela} and \ref{hopf}, and since there are no compact $\cH$-hypersurfaces in $\R^{n+1}$ if $\cH\in C^1(\S^n)$ vanishes somewhere (see \cite[Proposition 2.6]{BGM}), we have:

\begin{cor}
Let $\cH\in C^1(\S^2)$ be \emph{even} (i.e. $\cH(x)=\cH(-x)$ $\forall x\in \S^2$) and rotationally symmetric, and let $\Sigma$ be an immersed $\cH$-surface in $\R^3$ diffeomorphic to $\S^2$. Then, $\Sigma$ is the rotational, strictly convex $\cH$-sphere $S_{\cH}$ given by Theorem \ref{dela}.
\end{cor}

\section{Examples with no CMC counterpart}\label{sec:nocmc}

The results in Sections \ref{sec:bowls} and \ref{sec:rot2} show, in particular, the existence of the following types of rotational $\cH$-hypersurfaces in $\R^{n+1}$ for rather general choices of the (rotationally symmetric) prescribed mean curvature function $\cH\in C^1(\S^n)$: hyperplanes, cylinders, bowls, catenoids, convex spheres, unduloids and nodoids. This might suggest that $\cH$-hypersurfaces behave roughly as CMC hypersurfaces. In this section we shall see, however, that there are many other possible types of rotational $\cH$-hypersurfaces in $\R^{n+1}$. For definiteness, we will restrict our discussion to the case $n=2$.

So, in what follows, we will consider some specific functions $\ch\in C^{\8}([-1,1])$, which define rotationally symmetric functions $\cH\in C^{\8}(\S^2)$ by means of \eqref{presim0}, and we will construct rotational $\cH$-surfaces $\Sigma$ in $\R^3$ associated to them.

\begin{figure}[h]
\begin{center}
\includegraphics[width=.09\textwidth,valign=c]{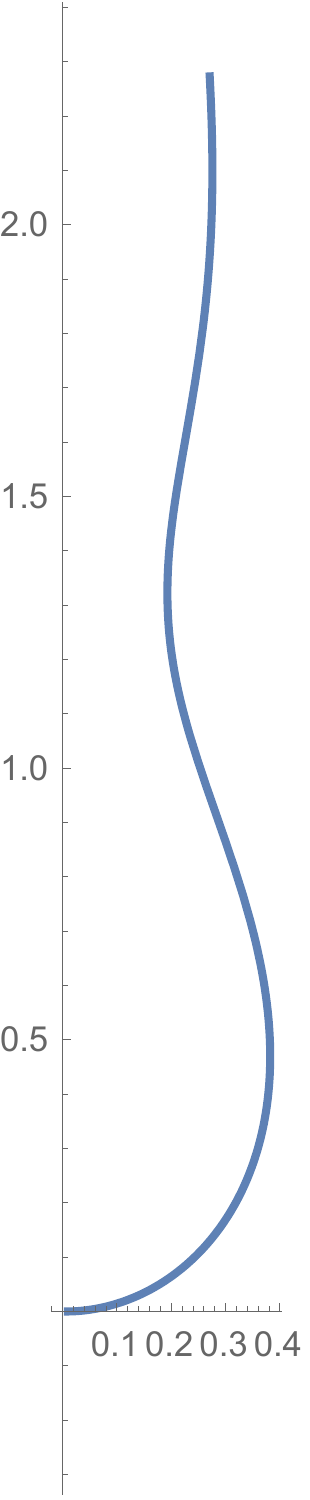}\quad
\hspace{2cm}
\includegraphics[width=.37\textwidth,valign=c]{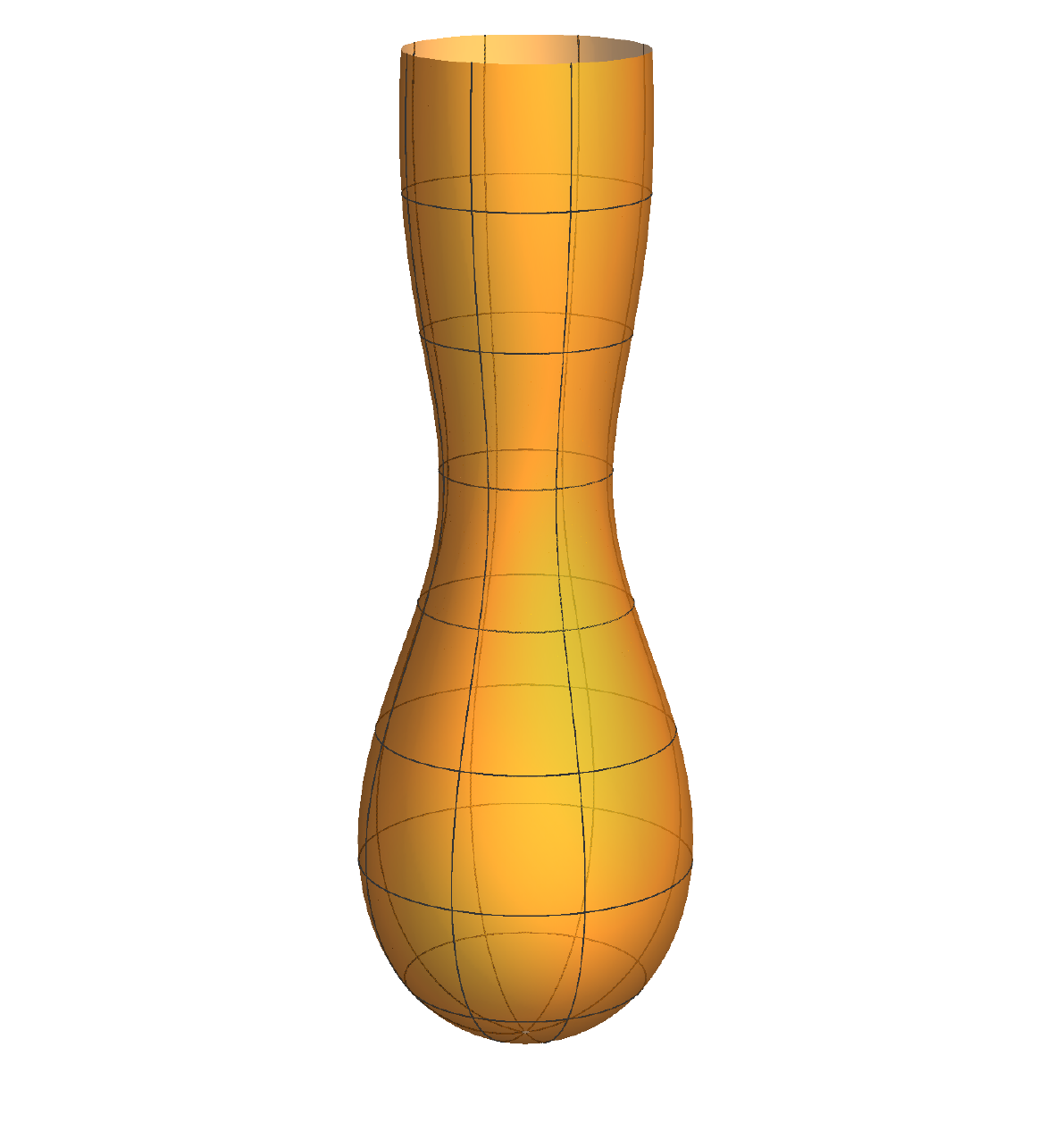}\quad 
\end{center}
\caption{Profile curve and picture for $\ch(y)=y+2$ of the rotational $\cH$-surface in $\R^3$ that meets its rotation axis with unit normal $e_3$.}
\label{fig:noes1}
\end{figure}

We first consider the case that $\Sigma$ touches its rotation axis orthogonally. If $\ch$ is an even positive function on $[-1,1]$, then $\Sigma$ is a strictly convex sphere (see Theorem \ref{dela}). However, if $\ch>0$ is not even, the generic situation is that one of the orbits that start at $(0,\pm 1)$ in the phase space $\Theta_1$ ends up spiraling around the equilibrium, while the other one ends up leaving $\Theta_1$ in a finite time of its parameter across one of the boundary curves $y=\pm 1$ (and thus, it enters the \emph{other} phase space $\Theta_{-1}$). In the first case, we obtain properly embedded rotational $\cH$-disks in $\R^3$ that converge asymptotically to the vertical cylinder of mean curvature $\ch(0)$, wiggling around it. See Figure \ref{fig:noes1}. In the second case, we typically obtain properly immersed, non-embedded, rotational $\cH$-disks that do not remain at a bounded distance from their rotation axis. See Figure \ref{fig:noes2}.

\begin{figure}[h]
\begin{center}
\includegraphics[width=.35\textwidth,valign=c]{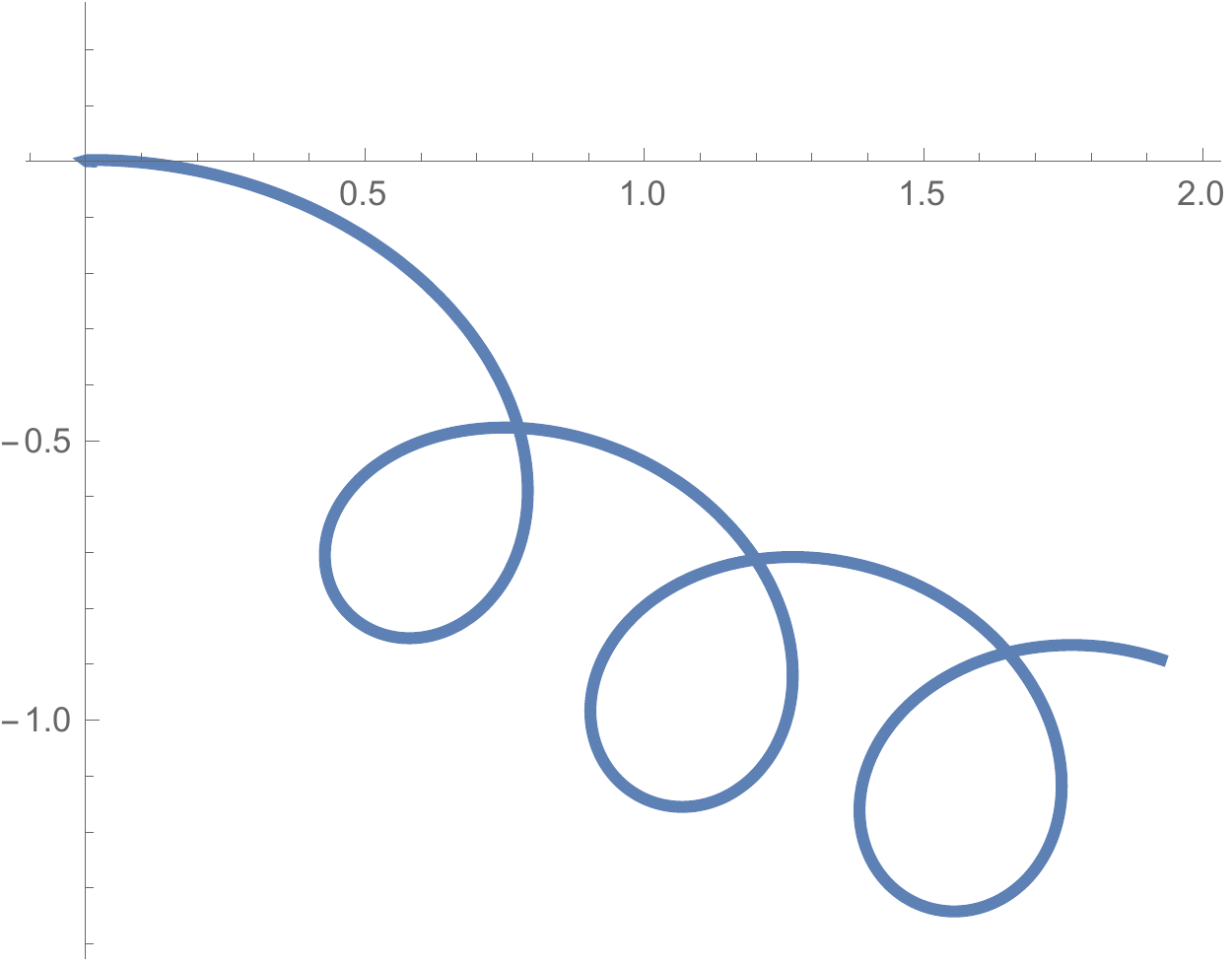}\quad 
\includegraphics[width=.6\textwidth,valign=c]{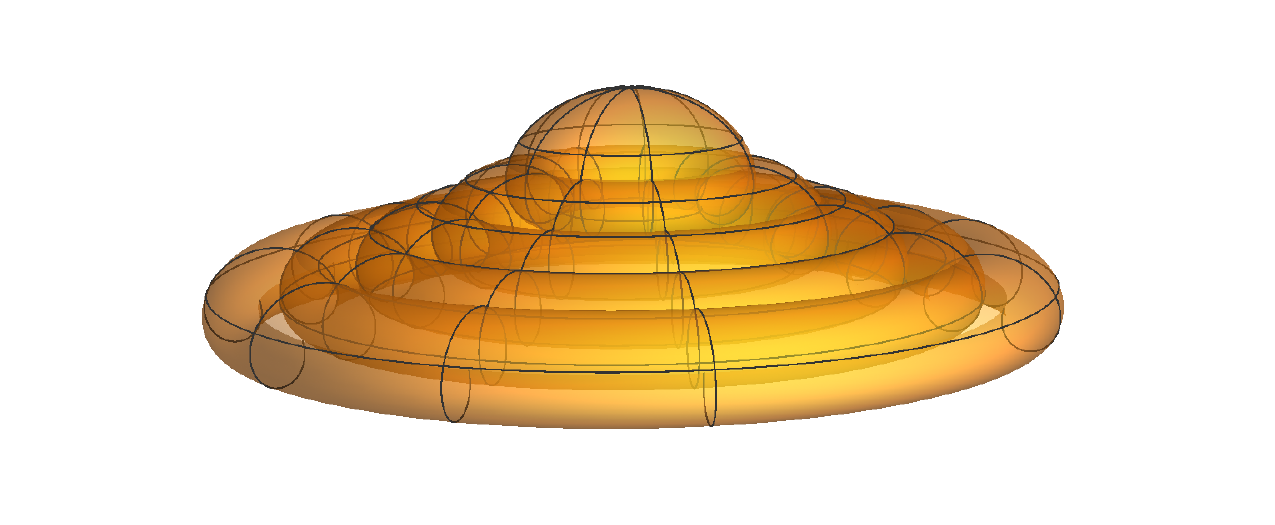}\quad 
\end{center}
\caption{Profile curve and picture for $\ch(y)=y+2$ of the rotational $\cH$-surface in $\R^3$ that meets its rotation axis with unit normal $-e_3$.}
\label{fig:noes2}
\end{figure}

Another less common, but still possible situation, is that the orbit in $\Theta_1$ that starts at $(0,\pm 1)$ converges to the equilibrium $e_0\in \Theta_1$ without spiraling around it. This gives examples of complete, non-entire, strictly convex rotational $\cH$-graphs that converge asymptotically to a cylinder. For instance, consider the rotational surface $\Sigma$ in $\R^3$ around the $x_3$-axis, with profile curve $\gamma(s)=(s,0,-\log(\cos s))$, with $s\in[0,\pi/2)$. Note that $\gamma(s)$ is the generating curve of the \emph{Grim reaper cylinder} in the theory of self translating solitons of the mean curvature flow. 

\begin{figure}[h]
\begin{center}
\includegraphics[width=.11\textwidth,valign=c]{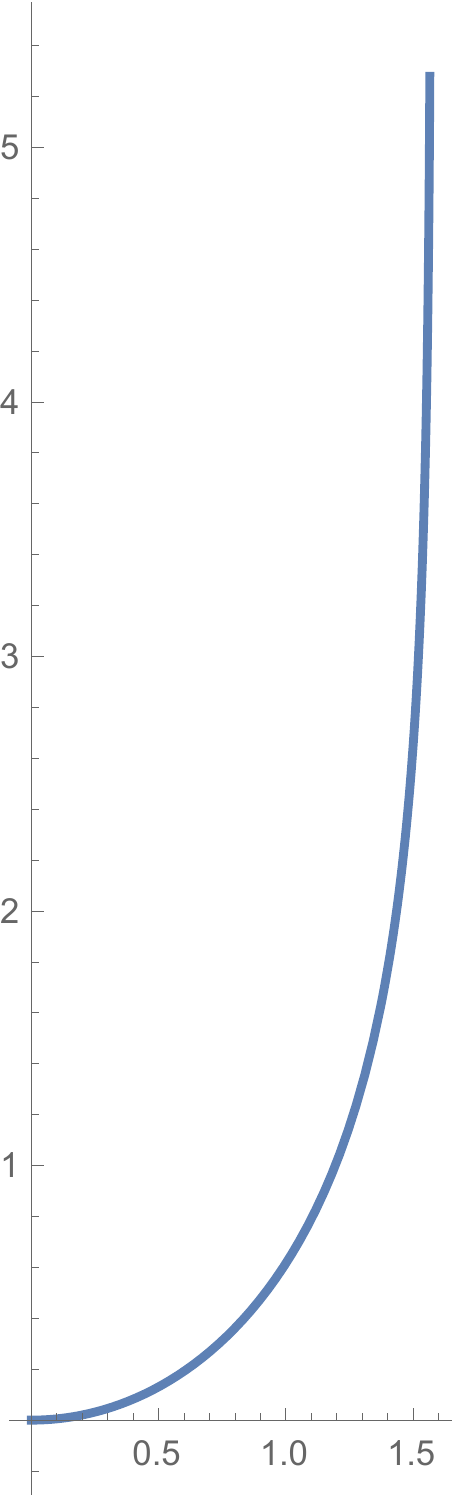}\quad
\hspace{3cm} 
\includegraphics[width=.27\textwidth,valign=c]{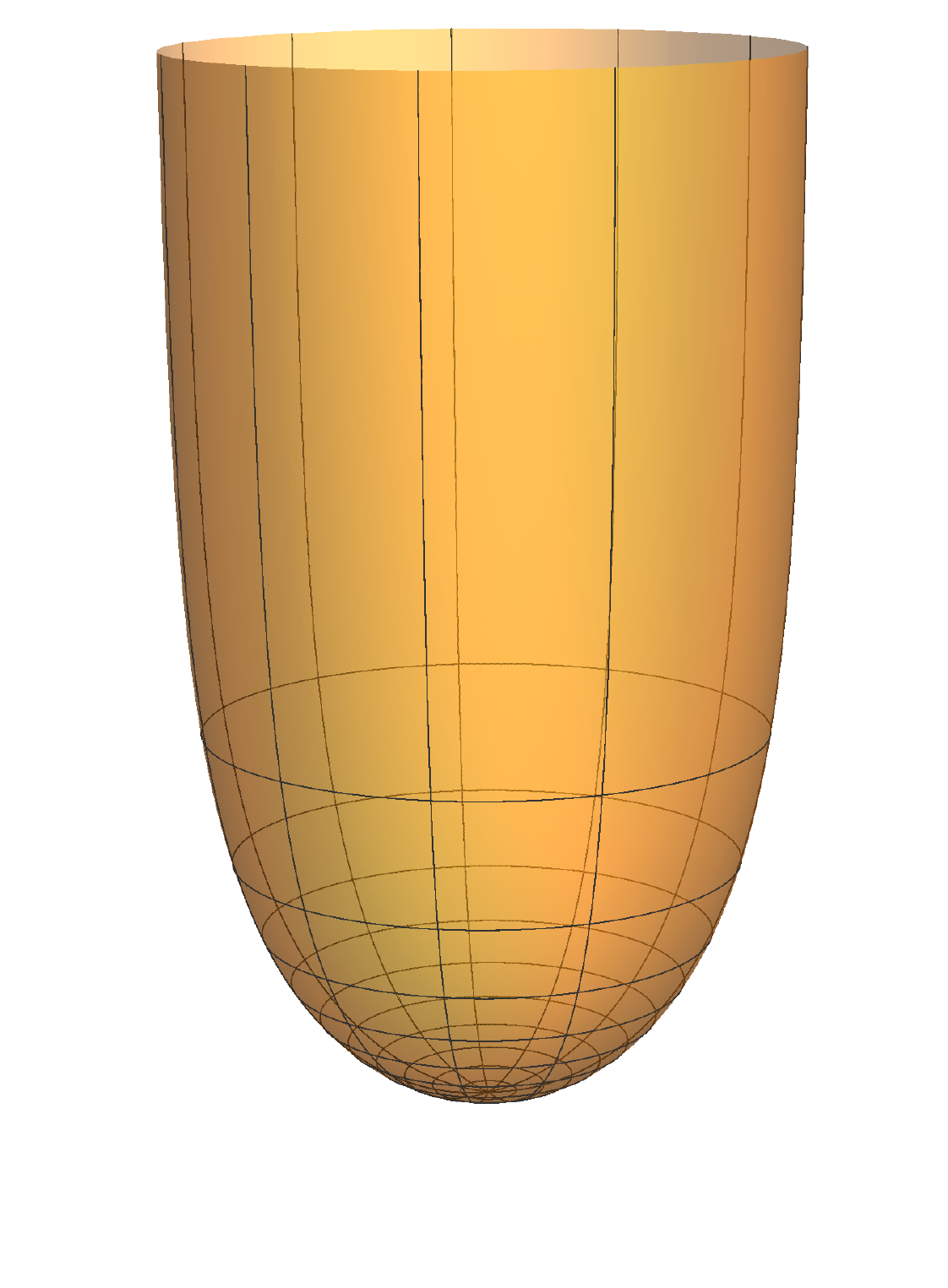}\quad 
\end{center}
\caption{A rotational, strictly convex, $\cH$-graph in $\R^3$ asymptotic to a cylinder, with mean curvature given by \eqref{hga}.}
\label{fig:grim}
\end{figure}

The curve $\gamma(s)$ is a strictly convex curve asymptotic to the vertical line $x=\pi/2$. A straightforward computation yields that $\Sigma$ has angle function given by $\nu(s)=\cos s$, and its mean curvature function can be written as a function of $\nu$, as follows: $H_{\Sigma}(s)=\ch_{\gamma}(\nu(s))$, where $\ch_{\gamma}\in C^1([0,1])$ is given by 
\begin{equation}\label{hga}
\ch_\gamma(y)=y+\frac{\sqrt{1-y^2}}{\arccos y}.
\end{equation}
Since $\ch_\gamma(y)>0$ for all $y\in [0,1]$, we can extend $\ch_\gamma$ to be a positive $C^1$ function on $[-1,1]$.
Thus, for the prescribed mean curvature function $\cH_{\gamma}>0$ in $\S^2$ associated to such $\ch_\gamma$ we obtain the existence of a strictly convex rotational $\cH_\gamma$-graph with the topology of a disk which is asymptotic to the right circular cylinder of radius $\pi/2$.

We consider next the case that the rotational surface $\Sigma$ does not touch its rotation axis. Assume that $\ch\in C^1([-1,1])$ satisfies $\ch(0)=0$, with $\ch(y)<0$ (resp. $\ch(y)>0$) if $y<0$ (resp. $y>0$). Then, by analyzing the resulting phase space as in previous results, we obtain the existence of \emph{wing-like $\ch$-catenoids}; see Figure \ref{fig:wing}. They have a similar shape to the well-known rotational wing-like translating solitons to the mean curvature flow (see \cite{CSS}), which are recovered for the specific choice $\ch(y)=y$. All of them are properly immersed annuli in $\R^3$, with their two ends going both either upwards or downwards. 

\begin{figure}[h]
\begin{center}
\includegraphics[width=.22\textwidth,valign=c]{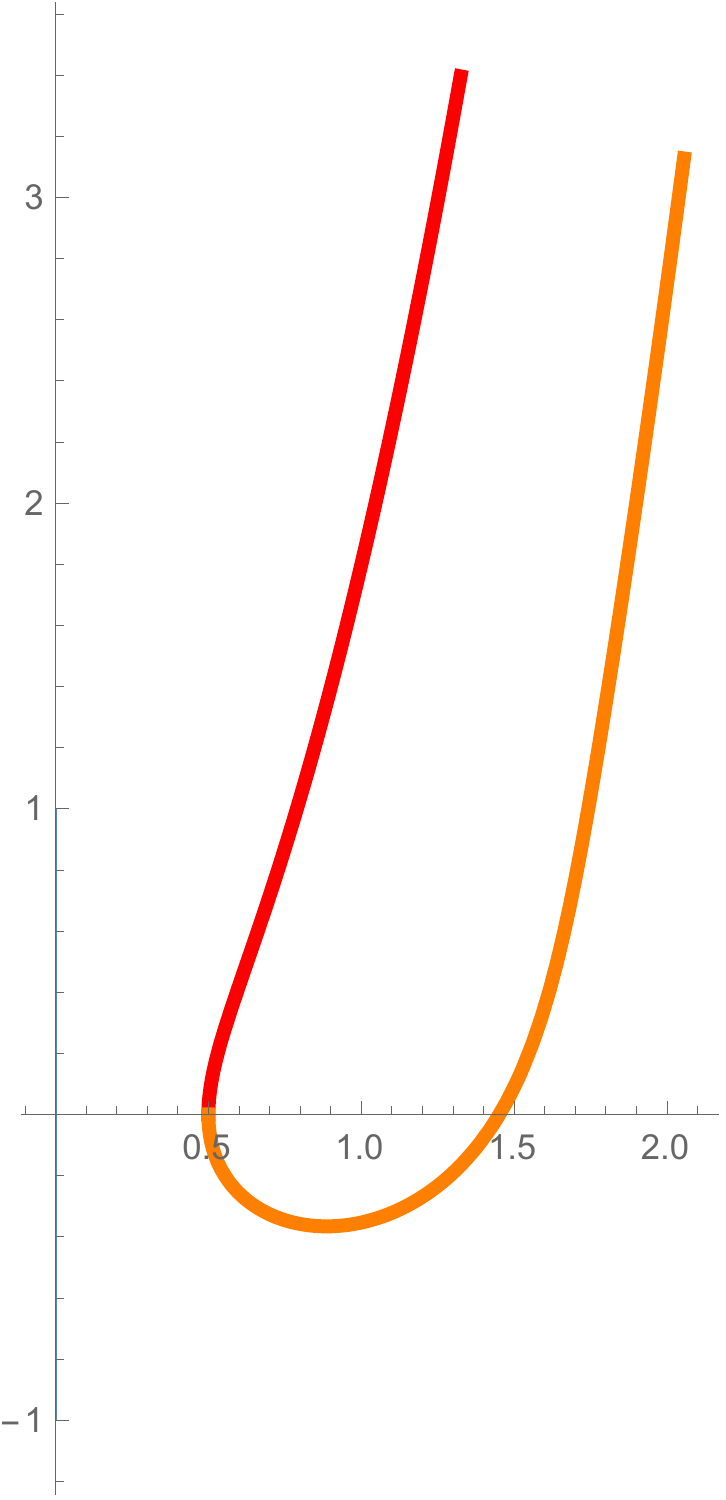}\quad 
\includegraphics[width=.43\textwidth,valign=c]{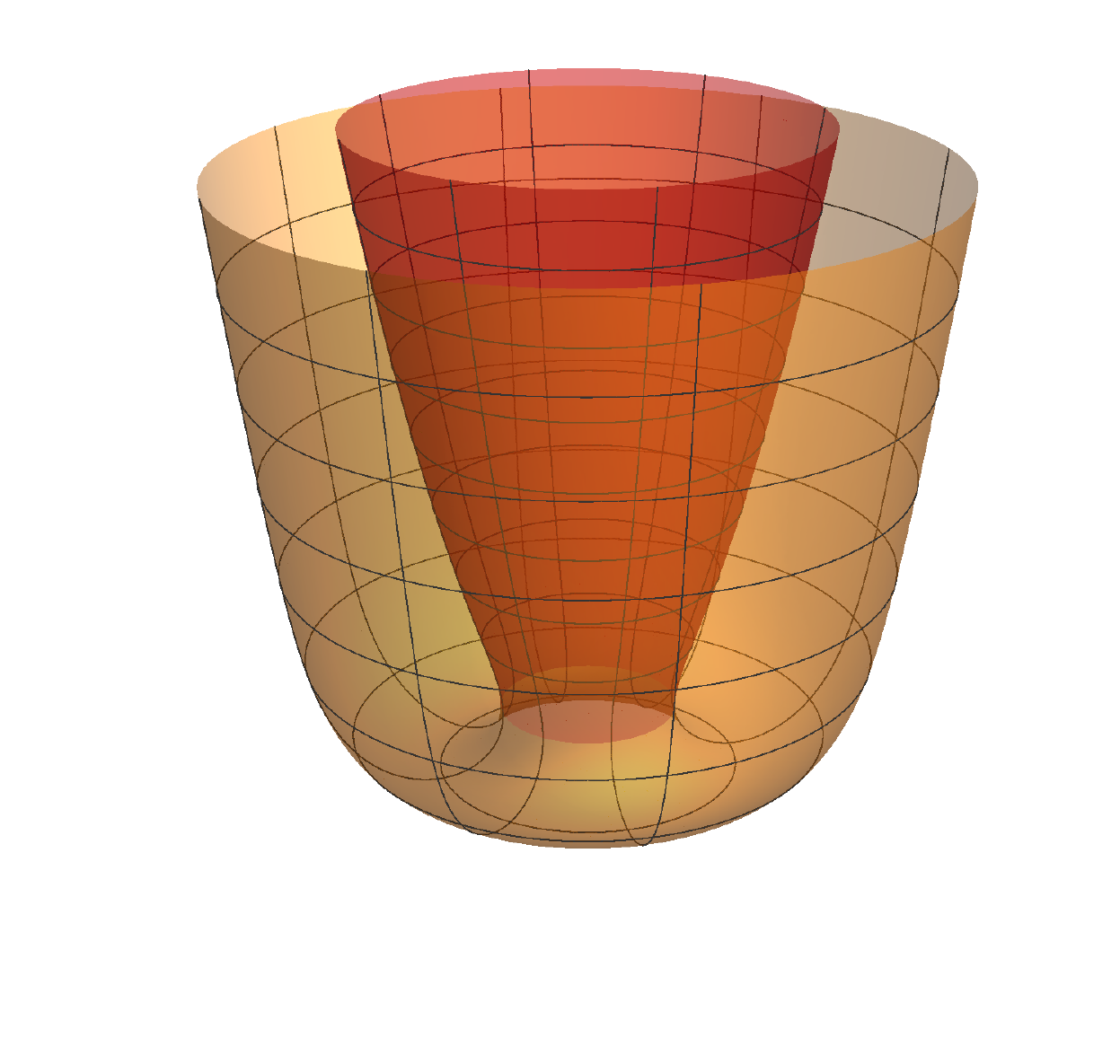}\quad 
\end{center}
\caption{A rotational wing-like $\cH$-surface in $\R^3$, for the choice $\ch(y)=y(y+2)$.}
\label{fig:wing}
\end{figure}

Consider, finally, the case that $\ch\in C^1([-1,1])$ satisfies $\ch(0)=0$, with $\ch(y)>0$ if $y\neq 0$. In that case there exist two types of properly immersed, rotational $\cH$-surfaces $\Sigma$ with the topology of an annulus, but this time with one end going upwards and the other going downwards. Each of such ends can be seen as a strictly convex graph over an exterior planar domain $\Omega=\R^2-D(0,r)$.

\begin{figure}[h]
\begin{center}
\includegraphics[width=.2\textwidth,valign=c]{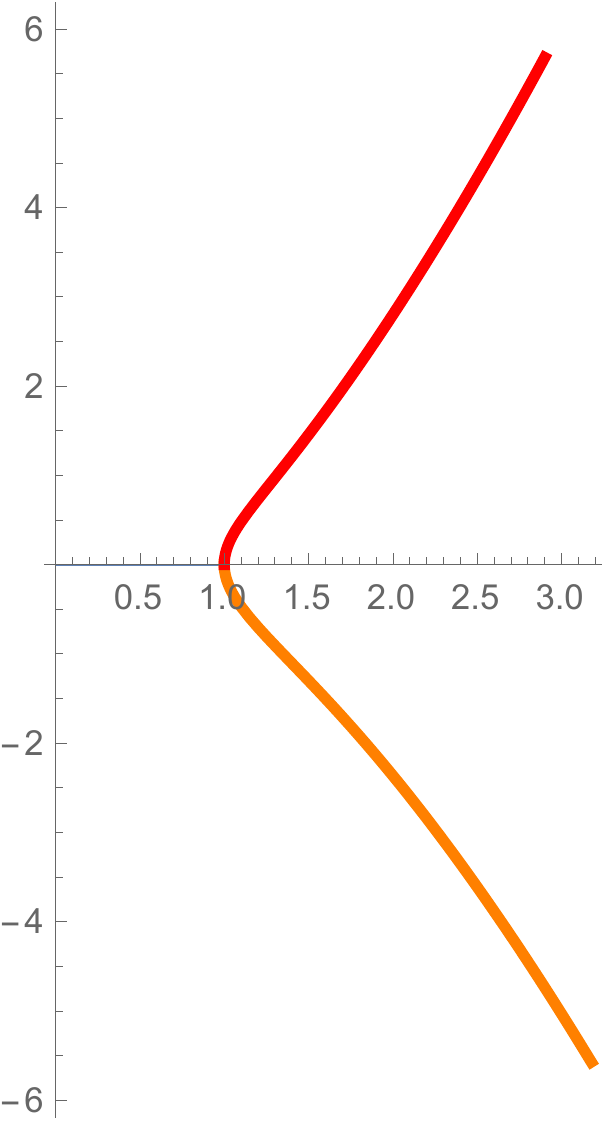}\quad
\hspace{2cm}
\includegraphics[width=.32\textwidth,valign=c]{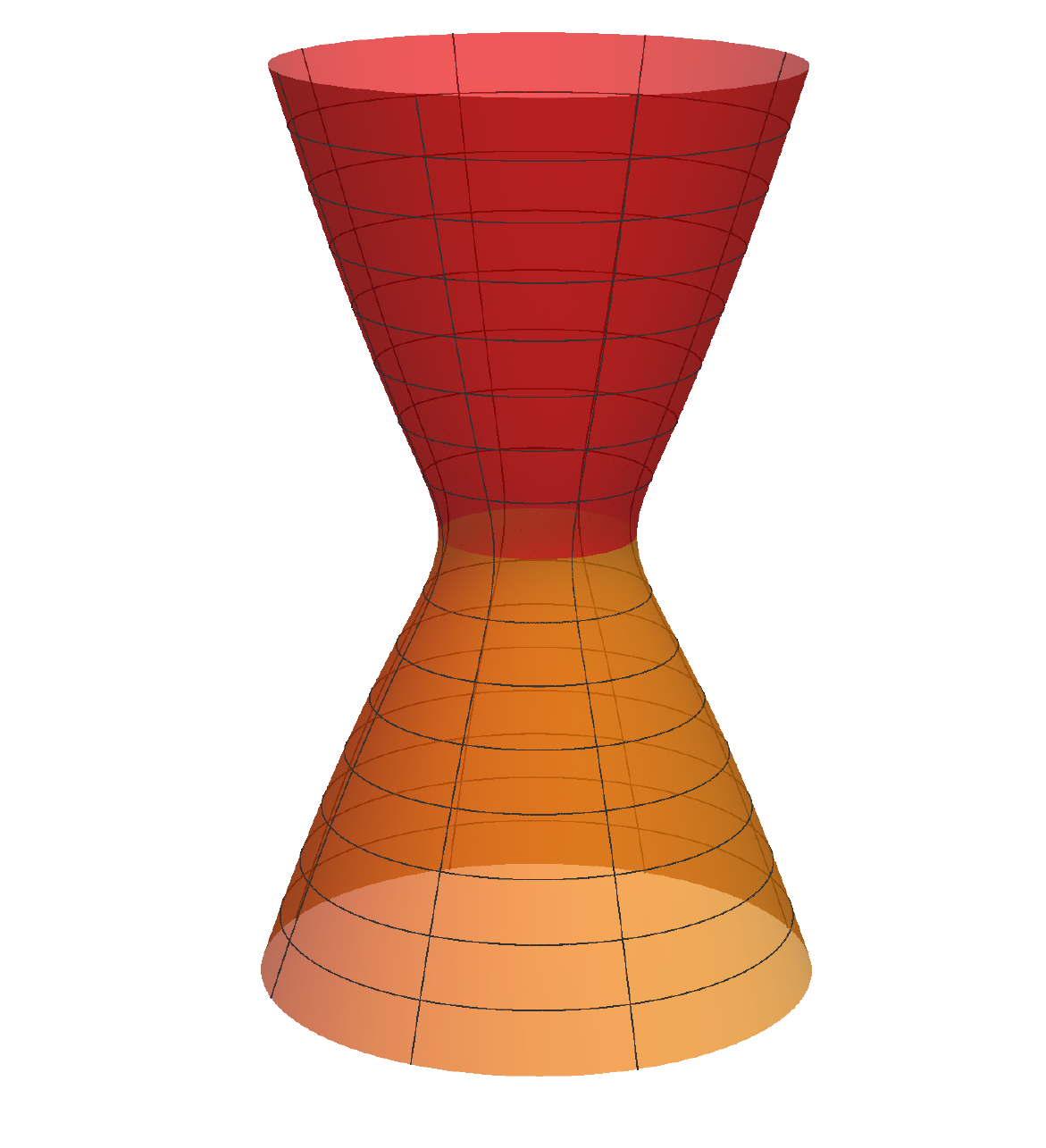}\quad 
\end{center}
\caption{A rotational embedded $\cH$-annulus in $\R^3$, for the choice $\ch(y)=y^2(y+2)$.}
\label{fig:hno1}
\end{figure}

In the first type of these examples, the coordinate $z(s)$ of the profile curve of $\Sigma$ is monotonous, and so $\Sigma$ is properly embedded; see Figure \ref{fig:hno1}. In the second type of examples, $z(s)$ is not monotonous, and the profile curve describes a \emph{nodoid-type} loop along which loses embeddedness. See Figure \ref{fig:hno2}

\begin{figure}[h]
\begin{center}
\includegraphics[width=.2\textwidth,valign=c]{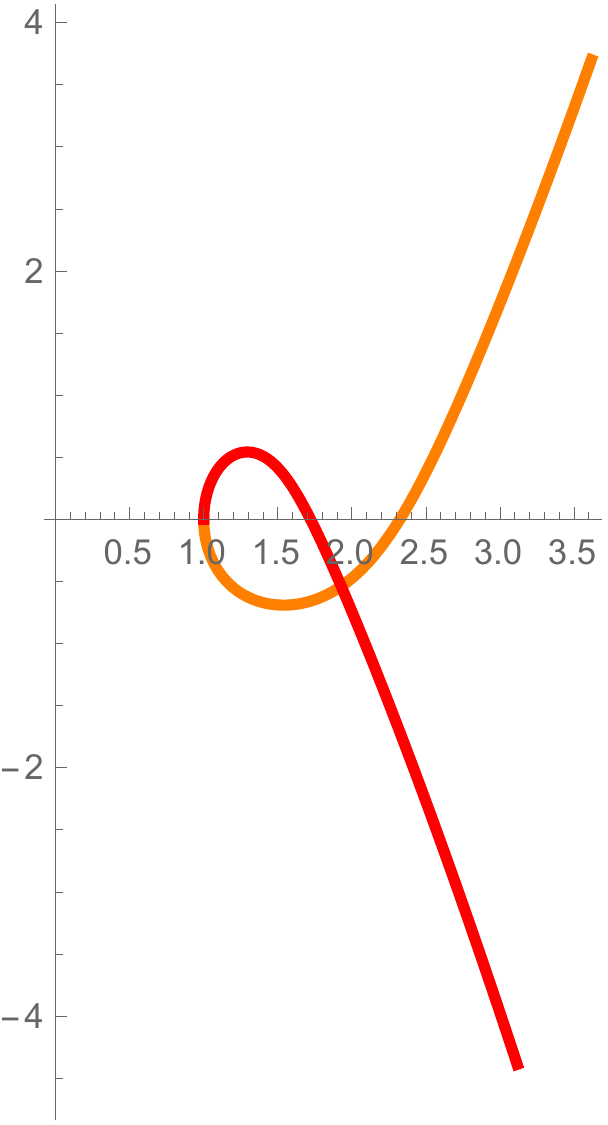}\quad
\hspace{2cm}
\includegraphics[width=.32\textwidth,valign=c]{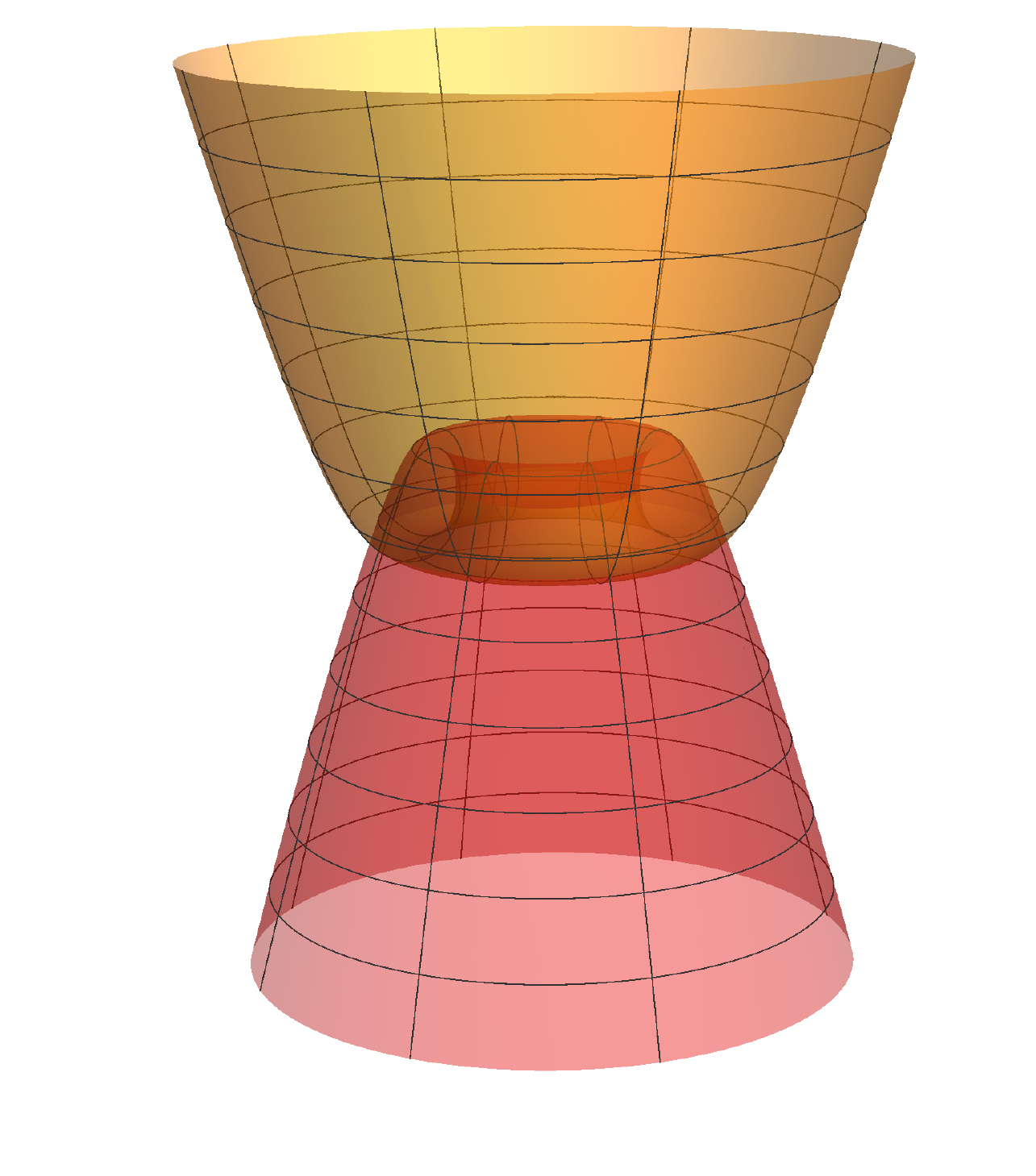}\quad 
\end{center}
\caption{A rotational non-embedded $\cH$-annulus in $\R^3$, for the choice $\ch(y)=y^2(y+2)$.}
\label{fig:hno2}
\end{figure}

\def\refname{References}

\hspace{0.4cm}

\noindent The authors were partially supported by
MICINN-FEDER, Grant No. MTM2016-80313-P, Junta de Andalucía Grant No.
FQM325, 
and Programa de Apoyo
a la Investigacion, Fundacion Seneca-Agencia de Ciencia y
Tecnologia Region de Murcia, reference 19461/PI/14.


\begin{thebibliography}{9}

\small{


\bibitem[Al]{Al} A.D. Alexandrov, Uniqueness theorems for surfaces in the large, I, {\it Vestnik Leningrad Univ.} {\bf 11} (1956), 5--17. (English translation: {\it Amer. Math. Soc. Transl.} {\bf 21} (1962), 341--354).

\bibitem[BV]{BV} M.F. Bidaut-Veron, Rotationally symmetric hypersurfaces with prescribed mean curvature, {\it Pacific J. Math.} {\bf 173} (1996), 29--67.

\bibitem[B1]{B1} A. Bueno, Translating solitons of the mean curvature flow in the space $\mathbb{H}^2\times \R$, {\it J. Geom.} (2018) 109:42. https://doi.org/10.1007/s00022-018-0447-x

\bibitem[B2]{B2} A. Bueno, Prescribed mean curvature surfaces in the product spaces $\mathbb{M}^2(\kappa)\times \R$; Analysis of rotational surfaces, existence of spheres and a Delaunay-type classification result, preprint (2018), arXiv:1807.10040.

\bibitem[B3]{B3} A. Bueno, Half-space theorems for properly immersed surfaces in $\R^3$ with prescribed mean curvature, preprint, (2019), arXiv:1901.04343 

\bibitem[BGM]{BGM} A. Bueno, J.A. Gálvez, P. Mira, The global geometry of surfaces with prescribed mean curvature in $\R^3$, preprint (2018), arXiv:1802.08146.

%
\bibitem[CSS]{CSS} J. Clutterbuck, O. Schnurer, F. Schulze, Stability of translating solutions to mean curvature flow, {\it  Calc. Var. Partial Diff. Equations} {\bf 29} (2007), no. 3, 281--293.

\bibitem[DRT]{DRT} D. de la Fuente, A. Romero, P.J. Torres, Radial solutions of the Dirichlet problem for the prescribed mean curvature equation in a Robertson-Walker spacetime, {\it Adv. Nonlinear Studies} {\bf 15} (2015), 171--181.

\bibitem[DG]{DG} M. del Pino, I. Guerra, Ground states of a prescribed mean curvature equation, {\it J. Differential Equations}, {\bf 241}
(2007), 112--129.

\bibitem[GM1]{GM1} J.A. Gálvez, P. Mira, A Hopf theorem for non-constant mean curvature and a conjecture of A.D. Alexandrov, {\it Math. Ann.} {\bf 366} (2016), 909--928.

\bibitem[GM2]{GM2} J.A. Gálvez, P. Mira, Uniqueness of immersed spheres in three-manifolds, {\it J. Diff. Geom.}, to appear. arXiv:1603.07153 


\bibitem[GM3]{GM3} J.A. Gálvez, P. Mira, Rotational symmetry of Weingarten spheres in homogeneous three-manifolds, preprint, arXiv:1807.09654


\bibitem[HM]{HM} D. Hoffman, W.H. Meeks, The strong halfspace theorem for minimal surfaces, {\it Invent. Math.} {\bf 101} (1990), 373--377.


\bibitem[K]{K} K. Kenmotsu, Surfaces of revolution with prescribed mean curvature, {\it Tohoku Math. J.} {\bf 32} (1980), 147--153.

\bibitem[KN]{KN} K. Kenmotsu, T. Nagasawa, Global existence of generalized rotational hypersurfaces with prescribed mean curvature in Euclidean spaces, I, {\it J. Math. Soc. Japan} {\bf 67} (2015), 1077--1108.

\bibitem[LRo]{LRo} C. Leandro, H. Rosenberg. Removable singularities for sections of Riemannian submersions of prescribed mean curvature, {\it Bull. Sci. Math.} {\bf 133} (2009), 445--452.

\bibitem[L]{Lo} R. López, Invariant surfaces in Euclidean space with a log-linear density, {\it Adv. Math.} {\bf 339} (2018), 285--309.

\bibitem[L2]{Lo2} R. López, Spacelike Graphs of Prescribed Mean Curvature in the Steady State Space, {\it Adv. Nonlinear Studies}, {\bf 16} (2019), 807--819. 

\bibitem[Mar]{Mar} T. Marquardt, Remark on the anisotropic prescribed mean curvature equation on arbitrary domains, {\it Math. Z.} {\bf 264} (2010), 507--511.


\bibitem[Po]{Po} A.V. Pogorelov, Extension of a general uniqueness theorem of A.D. Aleksandrov to the case of nonanalytic surfaces (in Russian), {\it Doklady Akad. Nauk SSSR} {\bf 62} (1948), 297--299.

\bibitem[P]{P} A. Pomponio, Oscillating solutions for prescribed mean curvature equations: euclidean and lorentz-minkowski cases, {\it Discrete Contin. Dyn. Syst. A} {\bf 38} (2018), 3899--3911.


}
\end{thebibliography}
\end{document}